\documentclass[12pt,reqno]{amsart}
\usepackage[utf8]{inputenc}
\usepackage{url}
\usepackage[colorlinks=true, allcolors=blue,linkcolor=blue, citecolor=blue]{hyperref}
\usepackage[capitalise]{cleveref}
\usepackage{amsfonts}
\usepackage{amsmath}
\usepackage{amssymb}
\usepackage{cleveref}
\usepackage[shortlabels]{enumitem}

\usepackage{amsthm}
\DeclareMathOperator\supp{supp}
\newtheorem{definition}{Definition}
\newtheorem{Lemma}{Lemma}
\newtheorem{Proposition}{Proposition}
\newtheorem{Theorem}{Theorem}
\newtheorem{Remark}{Remark}
\newcommand{\re}{\mathbb{R}}
\newcommand{\veps}{\varphi}
\newcommand{\fe}{f_{\epsilon}}
\newcommand{\pa}{\partial}
\newcommand{\nt}{\widetilde{\eta}}
\newcommand{\sign}{\text{sign}}
\usepackage[mathscr]{euscript}
\DeclareSymbolFont{rsfs}{U}{rsfs}{m}{n}
\DeclareSymbolFontAlphabet{\mathscrsfs}{rsfs}

\title{A Note on $L^1-$contractive property of the solutions of the scalar conservation laws through the method by Lax-{O}le\u{\i}nik. }
\author{Abhishek Adimurthi.}
\address{Department of Mathematics, Indiana University, 
Bloomington, IN 47405, USA}
\email{abadim@iu.edu,abhishek.adimu@gmail.com}

\begin{document}
\nocite{*}

\begin{abstract}
    In this note, we study the $L^1-$contractive property of the solutions the scalar conservation laws, got by the method of Lax-{O}le\u{\i}nik. First, it is proved when f is merely convex and the initial data is in $L^{\infty}(\re)$. And then, it is shown for the case when the initial data is in $L^1(\re)$ with the convex flux having super-linear growth. Finally, the $L^1-$contractive property is shown for the scalar conservation laws with the initial data in $L^1(\re)$ and the flux is ``semi-super-linear". This entire note does not assume any results mentioned through the approach by Kruzkov.
\end{abstract}

\maketitle
\pagestyle{myheadings}
\markright{A Note on $L^1$ contraction of Lax-{O}le\u{\i}nik solution of the SCL.}

\section*{\bf Introduction.}

Let $f : \re \mapsto \re $ be real-valued function , $u_0 \in L^{\infty}(\re)$, for which let $u$ in $L^{\infty}(\re \times (0,\infty))$ be a weak solution of the scalar conservation law,
\begin{equation}\label{SC1}
\left\{ 
\begin{aligned}
    \frac{\pa}{\pa t}u + \frac{\pa}{\pa x}\big[f(u)\big] &= 0 ; \quad x \in \re, t>0, \\
    u(x,0) &= u_0(x) ; \quad x \in \re,
\end{aligned} 
\right.
\end{equation}

i.e, in the weak sense, one can write (\ref{SC1}) as the following integral system:

\begin{align}
\label{weak}
\int_0^{\infty}\int_{-\infty}^{+\infty} \left(u \varphi_t + f(u) \varphi_x \right) dx dt + \int_{-\infty}^{\infty}u_0(x) \varphi(x,0)\  dx = 0,
\end{align}
for all test functions $\varphi \in C_c^{\infty}(\re \times [0,\infty))$, with $\varphi_t = \frac{\pa}{\pa t} \varphi$ and $\varphi_x = \frac{\pa}{\pa x} \varphi$.
In general, (\ref{weak}) can admit many solutions. A question of interest to ask here is for what set of functions does (\ref{weak}) admit a unique solution?

It is mentioned in \cite{Evans} that if the function $f$ is taken to be uniformly convex, i.e $f \in C^2(\re)$ and there exists $C>0$ such that $f''(x) \geq C$ , for all $x \in \re$, then by the \cite{Lax} and \cite{Oleinik}, an explicit solution is obtained by looking at a corresponding Hamilton Jacobi Equation.

The explicit formula then gives {O}le\u{\i}nik-one-sided inequality : 
\begin{equation}\label{OOS}
    u(x+z,t) - u(x,t) \leq C( 1 +t^{-1})z,
\end{equation}
for some $C \geq 0$ and for a.e $x \in \re$ , $t>0$, $z >0$.

It is shown in \cite{Oleinik} that there exist a unique solution in terms of (\ref{weak}), i.e if $u_1$ and $u_2$ satisfy the {O}le\u{\i}nik-one-sided inequality and have the same initial condition, then $u_1 = u_2$ a.e.

On the other hand, using the vanishing viscosity method, it is proved in \cite{Kruzkov} that the PDE (\ref{SC1}) attains a weak solution and satisfy certain integral inequalities as mentioned in the \cref{defn1}.

\begin{definition}\label{defn1}
  Fix $T > 0$, for which define $\pi_T := \re \times (0,T)$. A bounded measurable function $u : \pi_T \mapsto \re $ is called generalised entropy solution (in the sense of Kru\v{z}kov) of the PDE (\ref{SC1}) if the following holds,
  \begin{itemize}
      \item For any constant $K \in \re$ and for any non-negative test function $\varphi \in C_c^{\infty}(\pi_T)$, there holds the inequality 
\begin{equation}
    \label{kruz1}
    \int_{\pi_T} \Big[|u-K| \varphi_t + \sign(u-k)\left[f(u) - f(K)\right]\varphi_x \Big]dx dt \geq 0.
\end{equation}
\item The function $u(t,.)$ converges to $u_0$ as $t \rightarrow 0^+$ in the topology of $L^1_{loc}(\re)$, i.e,

\begin{equation}
\forall [a,b]\subset \re, \quad \lim_{t\rightarrow 0^+}\int_a^b | u(x,t) - u_0(x)| dx = 0.
\end{equation}

  \end{itemize}
 \end{definition}
  
 It's shown in \cite{Kruzkov} that if $u_1$ and $u_2$ are two generalised entropy solution (in the sense of Kru\v{z}kov), with initial data $u_{10}$ and $u_{20}$, then for a.e $t>0$, there holds that 
\begin{equation}
    \label{eq6}
    \forall a<b, \quad \int_a^b |u_1(x,t) - u_2(x,t)| dx \leq \int_{a-Lt}^{b+Lt}|u_{10}(x) - u_{20}(x)| dx,
\end{equation}
  where \[L:= \lim_{\epsilon\rightarrow 0}\text{essup}\Big\{ |f'(p)| ; p \in I_{\epsilon}\Big\},\] with
  \[I_{\epsilon} := [-max (\lVert u_{10} \rVert_{\infty} , \lVert u_{20} \rVert_{\infty}) - \epsilon , max (\lVert u_{10}\rVert_{\infty} , \lVert u_{20}\rVert_{\infty}) + \epsilon]. \]
\begin{definition}\label{defn2}
  The property mentioned in \cref{eq6} above is referred to as the $L^1$ contractive property of solutions.
\end{definition}

The approach in \cite{Kruzkov} has advantages over the one by Lax-{O}le\u{\i}nik as,
\begin{itemize}
    \item In \cite{Kruzkov}, $f$ is assumed to be just a  local lipshitz function.
    \item The method mentioned in \cite{Kruzkov} works in any dimension.
\end{itemize}

In this note, we always will assume that $f$ is convex. It is shown in \cite{Hoff} (see also \cite{MR2169977}, \cite{MR1304494}, \cite{MR688146}) that if the function $f$ is uniformly convex and $C^4$, then there exists $C>0$ such that for a.e $x \in \re,z >0, t>0$, there holds,
\[
u(x+z,t) - u(x,t) \leq  \frac{Cz}{t}, 
\]
where $u$ is a solution obtained by \cite{Kruzkov}. By the uniqueness result by Oleinik, the solution obtained by the method of \cite{Kruzkov} and by the method of \cite{Oleinik} are the same. Furthermore, using Oleinik's idea, \cite{Hoff} proves that if $f$ is $C^1$ and strictly convex, then 
for a.e $x\in \re, z>0, t>0$, there holds
\begin{equation}\label{olp1}
f'\left(u(x+z,t)\right)- f'\left(u(x,t)\right)\leq \frac{z}{t}.
\end{equation}
Moreover, \cite{Hoff} shows that if $u$ and $w$ are two weak solutions to the scalar conservation laws with the same initial data and satisfy \cref{olp1} for $f$ to be $C^1$ and strictly convex, then $u= w$ for a.e $(x,t) \in \re\times(0,\infty)$.

However, in this note, we look into answering the following questions with the assumption that the flux function $f$ is just convex :
\begin{enumerate}
    \item \label{Q2} Assume that the initial data $u_0$ is in $L^{\infty}(\re)$. Suppose the regularity condition on the function $f$ is relaxed i.e. there is no assumption made on the differentiability of the function $f$, does the method by Lax-{O}le\u{\i}nik through Hamilton Jacobi system provide a weak solution to the PDE (\ref{SC1})?
    \item\label{Q4} Suppose $u_{10}, u_{20}$ are in $L^{\infty}(\re)$. Consider the scalar conservation law with two different initial conditions, 
    \begin{equation}\label{SC2.1}
\left\{ \begin{aligned}
    \frac{\pa}{\pa t}u + \frac{\pa}{\pa x}\big[f(u)\big] &= 0 ; \quad x \in \re, t>0, \\
    u(x,0) &= u_{10}(x) ; \quad x \in \re,
\end{aligned} \right.
\end{equation}

\begin{equation}\label{SC2.2}
\left\{ \begin{aligned}
    \frac{\pa}{\pa t}u + \frac{\pa}{\pa x}\big[f(u)\big] &= 0 ; \quad x \in \re, t>0, \\
    u(x,0) &= u_{20}(x) ; \quad x \in \re.
\end{aligned} \right.
\end{equation}
Let the weak solutions (satisfying \cref{weak}) to the scalar conservation laws \cref{SC2.1} and \cref{SC2.2} be denoted by $u_1$ and $u_2$ respectively. Can we obtain $L^1$ contractive property for the solutions of these two scalar conservation laws, obtained through the method of Lax-{O}le\u{\i}nik? That is, for a.e $t>0$, is \cref{eq6} true i.e
    \[ \int_a^b |u_1(x,t) - u_2(x,t)| dx \leq \int_{a-Lt}^{b+Lt} |u_{10}(x) - u_{20}(x) | dx? \]
    \item \label{Q3} Now, suppose that $u_{10}, u_{20}$ are in $L^1(\re)$. Also, assume that the flux $f$ is super linear. Consider the scalar conservation law with two different initial conditions,
        \begin{equation}\label{SC3.1}
\left\{ \begin{aligned}
    \frac{\pa}{\pa t}u + \frac{\pa}{\pa x}\big[f(u)\big] &= 0 ; \quad x \in \re, t>0, \\
    u(x,0) &= u_{10}(x) ; \quad x \in \re,
\end{aligned} \right.
\end{equation}

\begin{equation}\label{SC3.2}
\left\{ \begin{aligned}
    \frac{\pa}{\pa t}u + \frac{\pa}{\pa x}\big[f(u)\big] &= 0 ; \quad x \in \re, t>0, \\
    u(x,0) &= u_{20}(x) ; \quad x \in \re.
\end{aligned} \right.
\end{equation}
Let the weak solutions (satisfying \cref{weak}) to the scalar conservation laws \cref{SC3.1} and \cref{SC3.2} be denoted by $u_1$ and $u_2$ respectively. Can we obtain $L^1$ contractive property for the solutions of these two scalar conservation laws, obtained through the method of Lax-{O}le\u{\i}nik?  Furthermore, for a.e $t>0$, does the weak solutions $u_1$ and $u_2$ satisfy 
    \[ 
    \int_{\re} |u_1(x,t) - u_2(x,t)| dx \leq \int_{\re} |u_{10}(x) - u_{20}(x) | dx ? \]

\item\label{Q5} Is the (Q.\ref{Q3}) true, when the flux is convex, but a relaxation is made on the super-linearity of the flux function $f$ and the function $f$ is assumed to be ``semi-super-linear" i.e.,
\[\lim_{p\rightarrow \pm\infty} \frac{f(p)}{p} = \mu_{\pm} \in [-\infty,+\infty], \text{ with } \mu_- \leq 0 \leq \mu_+?
\]

\end{enumerate}

In this note, we show that there are (weak) solutions that answer questions (Q.\ref{Q2}), (Q.\ref{Q4}),  (Q.\ref{Q3}) and (Q.\ref{Q5}) affirmatively (see \cref{thm1}, \cref{thm2} and \cref{thm3}). For the scalar conservation law with the initital data taken to be in $L^{\infty}(\re)$, the \cref{remark1} mentioned below tells that if the function $f$ is just taken to be convex, we can as well assume that $f$ can be convex and super-linear.  There are three main theorems mentioned in this note. The \cref{thm1} tells about establishing the $L^1$ contractivity for a scalar conservation law with the flux being just convex and super-linear and the initial data is in $L^{\infty}(\re)$. And therefore, $L^1-$contractive property holds when the flux is assumed to be just convex and the initial data is in $L^{\infty}(\re)$, by the \cref{remark1}. Then, the \cref{thm2} tells that similar results can be established for the initial conditions in the space $L^1(\re)$, but for the flux to be taken as super-linear and convex. Finally, a similar set of results is proved in \cref{thm3} for the case when the flux is convex, but a relaxation is made on the super-linearity of the flux function $f$ i.e. 
\[\lim_{p\rightarrow \pm\infty} \frac{f(p)}{p} = \mu_{\pm} \in [-\infty,+\infty], \text{ with } \mu_- \leq 0 \leq \mu_+. \]

 The approach in this note does not use any results from \cite{Kruzkov}. Moreover, either in \cite{Hoff} or in \cite{Oleinik}, $L^1-$contraction property for the solutions is not shown. So, independent to \cite{Kruzkov}, we plan on proving the $L^1$ contraction property for the scalar conservation laws with the flux function $f$ to be convex and having no conditions on it's regularity and thereby, establishing uniqueness.

In order to state the main results, we mention some notations.

\section*{\bf Notations.}\label{prelim}

For $u_0 \in L^{\infty}(\re)$, we define $v_0$ to be the primitive of $u_0$ as
\begin{align}\label{eq8}
    v_0(x) := \int_0^x u_0(t) \ dt.
\end{align}
Clearly, the function $v_0$ satisfy 
\begin{equation}
    |v_0(x) - v_0(y)| \leq \lVert u_0 \rVert_{\infty} |x-y|,
\end{equation}
and hence, $v_0$ is a lipshitz function with lipshitz constant $lip(v_0) \leq \lVert u_0\rVert_{\infty}$.

Let $f^*$ denote Fenchel dual of $f$, which is given by
\begin{equation}\label{eq11}
    f^*(q) := \sup \{ pq - f(p) ; p \in \re\}.
\end{equation}

Let $\eta_{\epsilon}$ be the mollifying sequence and define \[
f_{\epsilon}(x) := f\ast \eta_{\epsilon}(x) = \int_{\re} f(y) \eta_{\epsilon}(x-y) dy.
\]

Let $ 0 \leq s < t$, for which we define some functions as follows :
\begin{align}\label{eq13}
    \begin{split}
        V(x,t) &:= \inf_{y \in \re}\left\{ v_0(y) + t f^*\left(\frac{x-y}{t}\right)\right\}, \\
        V(x,s,t) &:= \inf_{y \in \re}\left\{ V(y,s) + (t-s) f^*\left(\frac{x-y}{t-s}\right)\right\}, \\
        Ch(x,s,t) &:= \left\{ \text{ all minimizers in the definition of $V(x,s,t)$ }\right\} ,\\
        Ch(x,t) &:= Ch(x,0,t).
    \end{split}
\end{align}
The function $V$ is called as the value function for the flux function $f$ and the set $Ch$ is called the charecteristic set.

For each $t>0$, $x \in \re$, define the functions $y_{+}$ and $y_-$ as
\begin{align}
    \begin{split}\label{eq1011}
        y_+(x,t) &:= \sup \{ y ; y \in Ch(x,t)\},\\
        y_-(x,t) &:= \inf \{ y ; y \in Ch(x,t)\}.
    \end{split}
\end{align}
\section*{\bf The Main Theorems.}

\begin{Remark}\label{remark1}
Owing to the \cref{thm1} mentioned below, if the function $f$ is convex and super-linear, we see that the solution is given by $\frac{\pa V(x,t)}{\pa x}$ which is bounded by $Lip(v_0)$, due to Rademacher's theorem. Therefore, noting that if $u_0 \in L^{\infty}(\re)$ is fixed and if $u$ is a weak solution as in (\ref{weak}) and $g:\re \mapsto \re$ be any continuous function such that $f(p) = g(p)$, for all 
\[
|p| \leq \lVert u \rVert_{\infty} \leq Lip(v_0) \leq \lVert u_0 \rVert_{\infty},
\]
then $u$ is also a weak solution of 
\begin{equation}\label{SC2}
\left\{ \begin{aligned}
    \frac{\pa}{\pa t}u + \frac{\pa}{\pa x}\big[g(u)\big] &= 0 ; \quad x \in \re, t>0, \\
    u(x,0) &= u_0(x) ; \quad x \in \re.
\end{aligned}\right.
\end{equation}

Hence, we can change the function $f$ which is just assumed to be convex, such that outside the interval , 
\[
[ -\lVert u_0 \rVert_{\infty} - 1,\lVert u_0 \rVert_{\infty} + 1],
\]
the ratio $f(p)/|p|$ blows up [refer (\ref{app1}$^{\text{st}}$) part of the Appendix]. Hence, we can assume that $f$ is convex and has super-linear growth, i.e,
\begin{equation}
    \lim_{|p| \rightarrow \infty} \frac{f(p)}{|p|} = \infty.
\end{equation}
The property of super-linearity of the function $f$ ensures that the Fenchel dual of $f$ is finite.
\end{Remark}

\begin{Theorem}\label{thm1} 
Let $u_0$ be a function in $L^{\infty}(\re)$. Define the primitive of $u_0$ as in \cref{eq8}. Let $f:\re \mapsto \re$ be a convex function. By the \cref{remark1}, we can assume that $f$ is convex and super-linear. Furthermore, define the Fenchel dual of $f$ as in \cref{eq11}. Also, define the value functions and the charecterstic sets as in \cref{eq13}. Then, there holds the following:
\begin{enumerate}
    \item The function $V$ is a lipshitz function with the property : 
    
    $\forall x, y \in \re , t>0$, we have
    \begin{equation}
        |V(x,t) - V(y,t)| \leq Lip(v_0)  |x-y|.
    \end{equation}
    \item The function $V$ satisfy the dynamic programming principle (ddp) i.e.
    \begin{equation}
        V(x,s,t) = V(x,t).
    \end{equation}
    \item The function $V$ is a viscosity solution of the Hamilton Jacobi system
    \begin{equation}
        \left\{\begin{aligned}
            V_t + f(V_x) &= 0, \quad x \in \re, t>0 \\
            V(x,0) &= v_0(x), \quad x \in \re.
        \end{aligned}\right.
    \end{equation}
    \item The function $u(x,t) := \frac{\pa}{\pa x}V(x,t)$ is a weak solution to the PDE (\ref{SC1}) with $\lVert u \rVert_{\infty} \leq \lVert u_0\rVert_{\infty }$.
    \item{$L^1$ Contractivity :} Let $u_{10} ,u_{20} \in L^{\infty}(\re)$ and set $L$ as in (\ref{eq6}).
    Let $u_1$ and $u_2$ be two weak solutions to the PDE (\ref{SC1}) with initial data $u_{10}$ and $u_{20}$, then for a.e $t>0$ and for $a < b$, we have
    \begin{equation}\label{eq17}
        \int_a^b |u_1(x,t) - u_2(x,t)|dx \leq \int_{a - L t}^{b+L t} |u_{10}(x) - u_{20}(x)|dx.
    \end{equation}
\end{enumerate}
\end{Theorem}
\begin{Remark}
This \cref{thm1} is different from the one mentioned in \cite{Evans}. In \cite{Evans}, the function $f$ is assumed to be uniformly convex and smooth. However, in the above \cref{thm1}, we just assume that $f$ is convex and has super-linear growth.
\end{Remark}

Note that if $u_0$ is in $L^{\infty}(\re)$, then the functions $u$ and $f(u)$ are in $L_{loc}^1(\re)$. However, in general, if $u_0 \in L^1(\re)$, apriori a weak solution $u$ of the PDE (\ref{SC1}) need not be well defined as $u$ and $f(u)$ need not be in $L^1_{loc}(\re)$. Also, the associated value function $V$ as defined in (\ref{eq13}) need not be lipshitz. So, we have the following definition: 

\begin{definition}\label{def3}
  The measurable function $u$ is said to be a weak solution to the PDE (\ref{SC1}) with it's initial value $u_0$ to be in $L^1(\re)$ if the following holds :
  \begin{itemize}
      \item The functions $u$ and $f(u)$ are in the space $L^1_{loc}(\re \times(0,\infty))$,
      \item The function $u$ satisfy the \cref{weak} in $\re\times(0,\infty)$ i.e., for all $\varphi \in C_c^{\infty}(\re \times (0,\infty)),$ there holds
      \[ \int_0^{\infty}\int_{-\infty}^{+\infty} \left(u \varphi_t + f(u) \varphi_x \right) dx dt  = 0 \]
      \item The function $u$ satisfy the equation, 
      \[  \forall [a,b]\subset \re, \quad \lim_{t\rightarrow 0^+}\int_a^b u(x,t) dx = \int_a^b u_0(x) dx. \]
  \end{itemize}
\end{definition}


\begin{Theorem}\label{thm2}
      Assume now that the function $f$ is convex and satisfy \[
      \lim_{|p|\rightarrow \infty}\frac{f(p)}{|p|} = \infty.
      \]Let $u_0 \in L^{\infty}(\re)$, $V$ be it's corresponding value function as defined in (\ref{eq13}) and set $u := \frac{\pa V}{\pa x}$. Then, from the \cref{thm1}, the function $u$ is a weak solution obtained from the Hamilton-Jacobi method. There holds the following :
      \begin{enumerate}
          \item\label{theorem2.1} (Comparison Principle.) For $u_{10}, u_{20} \in L^{\infty}(\re)$, let $u_1$, $u_2$ be the respective solutions obtained from the Hamilton-Jacobi method. Then, for a.e $x \in \re$, a.e $t \in (0,\infty)$, there holds the implication,
          \begin{equation}
              u_{10} (x) \leq u_{20}(x) \implies u_1(x,t) \leq u_2(x,t)
          \end{equation}
      \item\label{prev_point} For $u_{10}, u_{20} \in L^{1}(\re)$, let $u_{10n}$, $u_{20n}$ be the respective sequence of functions in $L^{\infty}(\re) \cap L^{1}(\re)$ such that $u_{i0n}(x)$ converges to $u_{i0}(x)$ in $L^1(\re)$, for $i=1,2$. Let $u_{1n}$, $u_{2n}$ be the corresponding solutions obtained from the Hamilton-Jacobi method for the initial data $u_{10n}$, $u_{20n}$ respectively. Then, for $i=1,2$, for $T>0$, we have
        \begin{itemize}
            \item $\{u_{1n}\} , \{u_{2n}\}$ are cauchy sequences in $L^1(\re\times(0,T))$.
            \item\label{it5022}  For $u_i$ to be the limit of $\{u_{in}\}$ in $L^1(\re\times(0,T))$ and for $0<\tau <T$, we have,
            \begin{equation}\label{eq319}
            \begin{split}
                \int_{-\infty}^{\infty}\int_{\tau}^T | u_1(x,t) &- u_2(x,t) | dx dt \\ &\leq (T-\tau)\int_{-\infty}^{\infty} \left|u_{10}(x) - u_{20}(x)\right| dx.
                \end{split}
            \end{equation}
        \item If for $i\in \{1,2\}$, the functions $\{v_{i0n}\}_{n \in \mathbb{N}}$ converges to $u_{i0}$ in $L^1(\re)$ with $v_{i0n} \in L^{\infty}(\re) \cap L^{1}(\re)$, then for a.e $(x,t)$ in $\re\times(0,\infty)$, we see that
        \begin{equation}\label{eq320}
            u_i(x,t) = v_i(x,t),
        \end{equation}
        where, $v_i := \lim_{n \rightarrow\infty}v_{in}$ in $L^1(\re \times (0,T))$, $\forall T>0$.
        \end{itemize}  
      \item Now, let $u_0 \in L^1(\re)$. As in the previous point (\ref{prev_point}), upon approximating $u_0$ by functions $\{ u_{0n} \in L^\infty \}$, one has the existence of solutions $\{ u_{n}\}$ for the scalar conservation law with the initial condition taken as $u_{0n}$. Take $u$ to be the limit of $\{u_{n}\}$ in $L^1(\re \times (0,T))$, $\forall T>0$, as in the previous point (\ref{prev_point}). Then, for any compact set $K \subset \re \times (0,\infty)$, we have
      \begin{itemize}
          \item The function $u$ is in $L^{\infty}(K)$.
          \item The function $u$ satisfy
          \[ \int_{0}^{\infty} \int_{-\infty}^{\infty}\Big[ u \varphi_t + f(u)\varphi_x \Big] (x,t) dx dt = 0, \quad \forall \varphi \in C_c^{\infty}(\re \times (0,\infty)). \]
        \item Furthermore, the function $u$ also satisfy,
          \[ \lim_{t\rightarrow 0 +} \int_a^b u(x,t) dx = \int_a^b u_0(x) dx , \quad \forall [a,b] \subset \re.\]
      \end{itemize}

    \end{enumerate}
\end{Theorem}



\newpage
\begin{Theorem}\label{thm3}
      Let $f:\re \mapsto \re$ be a convex function such that 
      \[\lim_{p \rightarrow \pm \infty}\frac{f(p)}{p} = \mu_{\pm},\]
      with $\mu_- \leq 0 \leq \mu_+$ and $u_0 \in L^1(\re)$. Then, there exist a weak solution $u \in L_{loc}^1(\re\times(0,\infty))$ to the PDE (\ref{SC1}) as mentioned in \cref{def3} with the initial data $u_0$. 
        Moreover, suppose that $u_{10}$ and $u_{20}$ are in $L^1(\re)$ and if the corresponding weak solutions are $u_1$ and $u_2$ in the space $L^1_{loc}(\re\times(0,\infty))$, then for a.e $t>0$, they satisfy 
      \begin{equation}\label{call3.1} \int_{-\infty}^{\infty}|u_1(x,t) - u_2(x,t)| dx \leq \int_{-\infty}^{\infty} |u_{10}(x) - u_{20}(x)| dx. 
      \end{equation}
     Furthermore, for any compact set $K \subset \re \times (0,\infty)$, we have
      \begin{itemize}
          \item The function $u$ is in $L^{\infty}(K)$.
          \item The function $u$ satisfy
          \[ \int_{0}^{\infty} \int_{-\infty}^{\infty}\Big[ u \varphi_t + f(u)\varphi_x \Big] (x,t) dx dt = 0, \quad \forall \varphi \in C_c^{\infty}(\re \times (0,\infty)). \]
        \item Finally, the function $u$ also satisfy,
          \[ \lim_{t\rightarrow 0 +} \int_a^b u(x,t) = \int_a^b u_0(x) dx , \quad \forall [a,b] \subset \re.\]
      \end{itemize}

\end{Theorem}

\section*{\bf Prerequisites for proving the Main Theorems.}
The idea of the proof(s) rely on the approach by \cite{Oleinik} and a stability result, along with some related lemmas which are stated below.

We start proving the main theorems by first assuming the following:
\begin{enumerate}
    \item\label{it21} $f : \re \mapsto \re$ is convex.
    \item $f$ has super-linear growth.
    \item\label{it23} For $\epsilon >0$, let $\eta_{\epsilon}$ be the mollifying sequence and define 
    \begin{equation}\label{molli1}
f_{\epsilon}(x) := f\ast \eta_{\epsilon}(x) = \int_{\re} f(y) \eta_{\epsilon}(x-y) dy.
    \end{equation}
    Then, we see the following :
    \begin{itemize}
    \item For every $\epsilon >0$, the functions $\{f_{\epsilon}\}$ are in $C^{\infty}(\re)$.
    \item\label{ab3} For $\epsilon >0$, $f_{\epsilon}$ is convex.
    \item The functions $\{f_{\epsilon}\}$ converges uniformly to $f$ on compact subsets of $\re$ as $\epsilon \rightarrow 0$.
    \item\label{it22} For every $0< \epsilon <1$, the functions $f_{\epsilon}$ has super-linear growth and the super-linear growth is uniform.
    \end{itemize}
\end{enumerate}

\begin{Lemma}\label{lem1}
Let $f$ be a real valued function which is convex and has super-linear growth. Also, define the mollified function $f_{\epsilon}$ as in \cref{molli1}. Then, there holds the following,
\begin{enumerate}
    \item $f^* : \re \mapsto \re$ is convex and has superlinear growth.
    \item As $\epsilon \rightarrow 0$, we see that the Fenchel dual of the mollified function $f^*_{\epsilon}$ goes to $f^*$ uniformly on compact sets.
    \item As $\epsilon \rightarrow 0$, we see that 
    \begin{equation*} \lim_{|q| \rightarrow \infty} \frac{f^*(q)}{|q|} = \infty, \quad \lim_{|q| \rightarrow \infty} \inf_{0 \leq \epsilon \leq 1}\frac{f^*_{\epsilon}(q)}{|q|} = \infty.
    \end{equation*}
    \item Let $\alpha \in C_c^{\infty}(B(0,1))$ be a non-negative function such that $$\int_{\re} \alpha(s) \  ds = 1.$$ For $\epsilon>0$, let $\left\{ \alpha_{\epsilon}(x) := \frac{1}{\epsilon}\alpha\left( \frac{x}{\epsilon}\right)\right\}$ be the mollifying sequence of the function $\alpha$. For every $F:\re \mapsto \re$, a convex function with super-linear growth, the function 
    \begin{equation}\label{mol}
        F_{\epsilon}(x) := \left(\alpha_{\epsilon}*F\right)(x) + \epsilon x^2,
    \end{equation}
    satisfy the properties (\ref{it21})-(\ref{it22}) mentioned in the assumptions. Moreover, the functions $ F_{\epsilon}$ is uniformly convex with respect to $\epsilon$.
 \end{enumerate}
\end{Lemma}
For the proof of the lemma (\ref{lem1}), refer (\ref{proflem1}$^{\text{nd}}$) part of the Appendix mentioned at the end of this note.
\begin{Lemma}\label{lem2}
Assume the hypothesis of the \cref{thm1}. Then, there holds the following. (see \cite{Conh}, \cite{Conh2}, \cite{Lax}).

\begin{enumerate}
    \item\label{lem2.1} The function $x \mapsto V(x,t)$ is a lipshitz function for all $t >0$ and we have 
    \begin{equation*}
        |V(x,t) - V(y,t)| \leq lip(v_0) |x-y|.
    \end{equation*}
    \item There exist $M>0$ such that \begin{equation*}
        V(x,t) = \inf \left\{ v_0(y) + tf^*\left(\frac{x-y}{t}\right) ;  \left| \frac{x-y}{t}\right| \leq M \right\}.
    \end{equation*}
    \item The set $Ch(x,s,t)$ is bounded and non-empty.
    \item We have the equality $V(x,s,t) = V(x,t)$, for all $0\leq s <t$. We also see that the function $(x,t) \mapsto V(x,t)$ is a lipshitz function with $V(x,0) = v_0(x)$.
    \item\label{it3005} Set $ u := \frac{\pa}{\pa x} V$. Then, $u$ is a weak solution of the PDE (\ref{SC1}) with $\lVert u\rVert_{\infty} \leq \lVert u_0\rVert_{\infty}$.
\end{enumerate}
\end{Lemma}

\begin{proof}
Plug $y=x$ in (\ref{eq13}) to get 
\begin{equation}\label{eq18}
    V(x,t) \leq v_0(x) + t f^*(0).
\end{equation}
From the property of super-linear growth of $f^*$, we can choose $q_0 >0$ such that for all $q \geq q_0 \geq 1$, we have $f^*(q) \geq \big[ lip(v_0) + 2 |f^*(0)|\big] |q|$ and so, for all $\left| \frac{x-y}{t}\right| \geq q_0$, we see that

\begin{align}
\begin{split}\label{eq19}
v_0 (y) + t f^*\left(\frac{x-y}{t}\right) & \geq v_0(x) + (v_0(y) - v_0(x)) + \big[lip(v_0) + 2|f^*(0)|\big] |x-y| \\
& \geq v_0(x) - lip(v_0) |x-y| + \big[lip(v_0) + 2|f^*(0)|\big] |x-y| \\
& = v_0(x) + 2 t |f^*(0)| \left| \frac{x-y}{t} \right| \\
& \geq v_0(x) + 2 t |f^*(0)|.
\end{split}
\end{align}

The inequalities (\ref{eq18}) , (\ref{eq19}) along with $\left| \frac{x-y}{t}\right| \geq q_0 \geq 1$, for $M = q_0$, we have 
\begin{equation}\label{eq20}
    V(x,t) = \inf \left\{ v_0(y) + tf^*\left(\frac{x-y}{t}\right) ;  \left| \frac{x-y}{t}\right| \leq M \right\}.
\end{equation}
The function $f^*$ is convex and so is continuous, which gives $Ch(x,t)$ to be nonempty and that infimum becomes minimum in (\ref{eq20}), i.e 
\begin{equation}\label{eq21}
    V(x,t) = \min \left\{ v_0(y) + tf^*\left(\frac{x-y}{t}\right) ;  \left| \frac{x-y}{t}\right| \leq M \right\}.
\end{equation}
So, for $x,z \in \re$, $y \in Ch(z,t)$, for all $\eta > 0$, we have 
\begin{equation}
    V(x,t) - V(z,t) \leq v_0(\eta) + t f^*\left(\frac{x-\eta}{t}\right) - v_0(y) - t f^*\left(\frac{z-y}{t}\right).
\end{equation}
Set $\eta = x - z + y$ to get
\begin{equation}\label{eq23}
     V(x,t) - V(z,t) \leq lip(v_0) |x-z|.
\end{equation}
Interchange $x \leftrightarrow z$ to obtain
\begin{equation}\label{eq24}
     |V(x,t) - V(z,t)| \leq lip(v_0) |x-z|,
\end{equation}
which proves the first two parts of the lemma and thus we have the function $V(x,s,t)$ to be lipshitz in $x-$variable with the estimates,
\begin{equation}\label{eq25}
    |V(x,s,t) - V(z,s,t)| \leq lip\big(V(.,s)\big) |x-z| \leq lip(v_0) |x-z|,
\end{equation}
\begin{equation}\label{eq26}
     V(x,s,t) = \inf \left\{ V(y,s) + (t-s)f^*\left(\frac{x-y}{t-s}\right) ;  \left| \frac{x-y}{t-s}\right| \leq M \right\}.
\end{equation}
The functions $f^*$ and $y \mapsto V(y,s)$ are continuous imply that the set $Ch(x,s,t)$ is non-empty and bounded, which concludes the third point of the lemma.

Define a new function $\gamma(\theta) := x + \left(\frac{x-y}{t-s}\right)(\theta - t)$, which satisfy the equality $\gamma(0) = \eta'$. Thus, $\eta'$ satisfy
\begin{equation}\label{eq27}
    \frac{x - \eta'}{t} = \frac{x-y}{t-s} = \frac{y - \eta'}{s},
\end{equation}
and so, we have
\begin{align}
    \begin{split}\label{eq28}
        V(x,s,t) & \leq V(y,s) + (t-s) f^*\left( \frac{x-y}{t-s}\right) \\
        & \leq v_0(\eta') + s f^*\left( \frac{y- \eta'}{s}\right) + (t-s) f^*\left( \frac{x-y}{t-s}\right) \\
        & = v_0(\eta') + s f^*\left( \frac{x- \eta'}{t}\right) + (t-s) f^*\left( \frac{x-\eta'}{t}\right)\\ 
        & = v_0(\eta') +  t f^*\left( \frac{x- \eta'}{t}\right).
    \end{split}
\end{align}
Taking the infimum over $\eta'$ gives 
\begin{equation}
    V(x,s,t) \leq V(x,t).
\end{equation}
To prove the other side of the inequality, as the sets $Ch(x,s,t)$ and $Ch(x,t)$ are non-empty, let $\alpha \in Ch(x,s,t)$ and $\beta \in Ch(\alpha,s)$. The convexity of $f^*$ along with the equality, \begin{equation}\label{eq30}\frac{x-\beta}{t} = \frac{x - \alpha}{t-s} \left(1 - \frac{s}{t}\right) + \frac{\alpha - \beta}{s} \left(\frac{s}{t}\right),
\end{equation} 
gives
\begin{equation}\label{eq31}
    t f^*\left(\frac{x-\beta}{t}\right) \leq (t-s)f^*\left(\frac{x - \alpha}{t-s}\right) + s f^*\left(\frac{\alpha - \beta}{s} \right).
\end{equation}
Hence, we have
\begin{align}
    \begin{split}\label{eq32}
        V(x,s,t) &= V(\alpha,s) + (t-s)f^*\left(\frac{x - \alpha}{t-s}\right) \\ 
        &= v_0(\beta) + s f^*\left(\frac{\alpha - \beta}{s} \right) + (t-s)f^*\left(\frac{x - \alpha}{t-s}\right) \\
        &\geq v_0(\beta) + t f^*\left(\frac{x-\beta}{t}\right) \\
        &\geq V(x,t),
    \end{split}
\end{align}
which concludes the fourth point of the lemma.

To prove the latter of the fourth point of the lemma, first observe that for $0 \leq t_1 < t_2$, for all $y \in \re$, there holds
\begin{equation}
    V(x,t_2) \leq V(x,t_1, t_2) \leq V(y,t_1) + (t_2 - t_1) f^*\left( \frac{x-y}{t_2 - t_1}\right).
\end{equation}
Setting $y =x$, we get
\begin{equation}\label{eq34}
    V(x,t_2) - V(x,t_1) \leq f^*(0) (t_2 - t_1).
\end{equation}
For $\nt \in Ch(x,t_2)$, we see that \begin{equation}\label{eq35} 
\left| \frac{x - \nt}{t_2}\right| \leq M
\end{equation}
and so, for all $y \in \re$, we get
\begin{equation}
\begin{split}
    V(x, t_2) - V(x,t_1) \geq v_0(\nt) &+ t_2 f^*\left(\frac{x - \nt}{t_2}\right) \\ &- v_0(y) -  t_1 f^*\left(\frac{x - y}{t_1}\right).
    \end{split}
\end{equation}
Choose $y$ such that \begin{equation}
    \frac{x - \nt}{t_2} = \frac{x-y}{t_1} \iff y-\nt = \frac{t_2 - t_1}{t_2} \left(x - \nt\right),
\end{equation}
so that along with \cref{eq35}, we get
\begin{equation}
    |y - \nt| \leq M|t_2 - t_1|.
\end{equation}
Hence, there holds
\begin{align}
    \begin{split}\label{eq39}
        V(x,t_2) - V(x,t_1) &\geq v_0(\nt) - v_0(y) + (t_2 - t_1) f^*\left( \frac{x - \nt}{t_2}\right) \\
        &\geq -lip(v_0) |\nt - y| - \lambda (t_2 - t_1),
    \end{split}
\end{align}
where, $\lambda := \sup \{ |f^*(z)| ; |z| \leq M\}$. Setting \begin{equation}
    C_1 := |f^*(0)| + lip(v_0) + \lambda,
\end{equation}
and along with \cref{eq34} and \cref{eq39}, we see that
\begin{equation}\label{eq41}
    |V(x,t_2) - V(x,t_1)| \leq C_1 |t_2 - t_1|.
\end{equation}
As a consequence, we get
\begin{align}
    \begin{split}
    |V(x_1,t_1) - V(x_2,t_2)| &\leq |V(x_1,t_2) - V(x_1,t_1)| + | V(x_1,t_2) - V(x_2,t_2)| \\ 
    &\leq lip(v_0)|x_1 - x_2| + C_1 |t_2 - t_1|,
    \end{split}
\end{align}
which concludes that the function $V$ is lipshitz continuous.

To prove the last point of the lemma, first observe that $V$ is a viscosity solution to the Hamilton Jacobi equation,
    \begin{equation}\label{eq43}
        \left\{\begin{aligned}
            V_t + f(V_x) &= 0, \quad x \in \re, t>0 \\
            V(x,0) &= v_0(x), \quad x \in \re.
        \end{aligned}\right.
    \end{equation}
    
The function $V$ is differentiable a.e and from the ``Touching by a $C^1$ function" lemma in \cite[Chapter~10]{Evans}, for a.e $(x,t) \in \re\times(0,\infty)$, the function $V$ satisfy the PDE (\ref{eq43}) point-wise. 

Now, choose $\varphi\in C_c^{\infty}(\re\times[0,\infty))$ and multiply (\ref{eq43}) by $\varphi_x$ to get
\begin{equation}\label{eq44}
    \int_0^{\infty} \int_{-\infty}^{\infty}\big[ V_t \varphi_x + f(V_x) \varphi_x\big] dx dt = 0.
\end{equation}
As the function $V$ is lipshitz, it is differentiable almost everywhere by the Rademacher's theorem and so we see that 
\begin{align}
    \begin{split}
        \int_0^{\infty} \int_{-\infty}^{\infty}V_t \varphi_x dx dt &= -  \int_{-\infty}^{\infty}V(x,0) \varphi_x(x,0) dx - \int_0^{\infty} \int_{-\infty}^{\infty}V(x,t) \varphi_{xt} dx dt \\
        &= \int_{-\infty}^{\infty}(v_0)_x \varphi(x,0)dx + \int_0^{\infty} \int_{-\infty}^{\infty}V_x(x,t) \varphi_{t} dx dt
    \end{split}
\end{align}
Finally, the \cref{eq44}, $u = \frac{\pa}{\pa x}V$ and $u_0(x) = \frac{\pa}{\pa x}v_0$ tells
\begin{equation}
    \int_0^{\infty} \int_{-\infty}^{\infty}\big[ u \varphi_t + f(u) \varphi_x \big] dx dt + \int_{-\infty}^{\infty} u_0(x) \varphi(x,0)dx = 0,
\end{equation}
and that
\begin{equation}
    \lVert u \rVert_{\infty} = \left\lVert \frac{\pa V}{\pa x}\right\rVert_{\infty} \leq lip(v_0) \leq \lVert u_0\rVert_{\infty}.
\end{equation}

\end{proof}

Next, we state a lemma based on \cite{AGM}.
\begin{Lemma}[Stability Result]\label{lem3}
    Let $\{\epsilon_n\}_{n \in \mathbb{N}}$ be a sequence going to $0$ and let the functions $f$ and $f_n := f_{\epsilon_n}$ satisfy the prerequisites  (\ref{it21}) - (\ref{it22}). Furthermore, for $u_0 \in L^{\infty}(\re)$, set $v_0$ to be the primitive (lipshitz) function i.e,
    \[
    v_0(x) = \int_0^x u_0(t)dt.    \]
    Also, let $V$ and $V_n$ be the corresponding value functions defined in (\ref{eq13}) for the flux $f$ and $f_n$ respectively. Then, we have the following results: 
 \begin{enumerate}
     \item We have that $V_n$ converges to $V$ uniformly on compact subsets of $\re \times [0,\infty)$ as $n \rightarrow \infty$.
     \item Let $0 \leq s < t$, $x \in \re$ and set $Ch_n(x,s,t)$ to be the charecteristic set related to $V_n$, $Ch(x,s,t)$ to be the charectersitic set relating $V$ as defined in (\ref{eq13}). 
     Set \begin{equation*}
         \lim_{n \rightarrow \infty} x_n = x, \quad \lim_{n \rightarrow \infty}y_n = y, \quad \lim_{n \rightarrow \infty}t_n = t, \quad \lim_{n \rightarrow \infty}s_n = s.
     \end{equation*}
     Then, for $y_n \in Ch_n(x_n,s_n ,t_n)$, we see that the point $y$ is in $Ch(x,s,t)$.
     \item\label{3oflem3} Let $u := \frac{\pa}{\pa x}V$ and set $u_n := \frac{\pa}{\pa x} V_n$. Then, for any $\varphi \in C_c^{\infty}(\re \times (0,\infty))$, we see that $u$ satisfy \cref{weak} i.e.
     \[
     \int_0^{\infty}\int_{-\infty}^{+\infty} \left(u \varphi_t + f(u) \varphi_x \right) dx dt + \int_{-\infty}^{\infty}u_0(x) \varphi(x,0)\  dx = 0,
     \]
     \[
     \text{and}
     \]
     \begin{equation*}
         \lim_{n \rightarrow \infty} \int_0^{\infty} \int_{- \infty}^{\infty} u_n \varphi \ dx dt = \int_0^{\infty} \int_{- \infty}^{\infty}u \varphi \ dx dt.
     \end{equation*}
 \end{enumerate}
\end{Lemma}

\begin{proof}
From the  assumptions (\ref{it21}) - (\ref{it22}), for all $n\in\mathbb{N}$, we see that
\begin{equation}
    \lim_{|q| \rightarrow \infty} \frac{f_n(q)}{|q|}  \geq \lim_{|q| \rightarrow \infty}\inf_j \frac{f_j(q)}{|q|}  =\infty.
\end{equation}
The proof in the lemma \ref{lem2} suggests that for the constant \[M:= lip(v_0) + 2 \sup_n|f_n^*(0)|,\] we have 
\begin{equation}
    V_n(x,t) = \inf \left\{ v_0(y) + tf_n^*\left(\frac{x-y}{t}\right) ;  \left| \frac{x-y}{t}\right| \leq M \right\}.
\end{equation}
The lemma \ref{lem1} tells that the sequence $\{f_n^*\}$ converges to $f^*$ uniformly on compact subsets of $\re$ and thus, there holds the statement :
\begin{equation}\label{eq50}
    \lambda := \sup_{n \in \mathbb{N}} \sup \{ |f_n^*(z)| ; |z| \leq M\} \text{ is bounded. }
\end{equation}
From \cref{eq41}, for $C_1 := \sup_n |f_n^*(0)| + lip(v_0) + \lambda$, we get
\begin{align}
    \begin{split}
        |V_n(x,t_1) - V_n(x,t_2)| &\leq C_1 |t_1 - t_2|,\\
        |V_n(x_1,t) - V_n(x_2,t)| &\leq lip(v_0) |x_1 - x_2|.
    \end{split}
\end{align}
The Arzela-Ascoli theorem gives the existence of a subsequence $\{V_{n_k}\}$ and a continuous function $V$ such that $V_{n_k}$ converges to $V$ uniformly on compact subsets.

Now, it suffices to show that the function $V$ is in fact the value function for the flux $f$. For $(x,t)\in\re\times(0,\infty)$, $y_n \in Ch_n(x,t)$, we have \[\left|\frac{x-y_n}{t}\right| \leq M,\] and so there is a subsequence $\{y_{n_k}\}$ converging to $y \in \re$. Thus, for $(z,t) \in \re\times(0,\infty)$, there holds
\begin{align}
    \begin{split}
        V(x,t) = \lim_{n_k \rightarrow \infty} V_{n_k}(x,t) &= \lim_{n_k \rightarrow \infty} \left[ v_0 (y_{n_k})  + t f^*_{n_k}\left( \frac{x - y_{n_k}}{t}\right)\right]\\
        &\leq \lim_{n_k \rightarrow \infty} \left[ v_0 (z)  + t f^*_{n_k}\left( \frac{x - z}{t}\right)\right],
    \end{split}
\end{align}
which along with the facts that the function $v_0$ being lipshitz continuous and the functions $f_n^*$ being uniformly continuous, implies
\begin{equation}
    V(x,t) \leq v_0(y) + t f^*\left(\frac{x-y}{t}\right) \leq v_0(z) + t f^*\left(\frac{x-z}{t}\right).
\end{equation}
So, we have
\begin{equation}
    V(x,t) = \inf \left\{ v_0(z) + tf^*\left(\frac{x-z}{t}\right) ;  \left| \frac{x-z}{t}\right| \leq M \right\},
\end{equation}
which is precisely the value function corresponding to $v_0$ and $f^*$ and this concludes the first part of the lemma.

For the second part of the lemma, for $z \in \re$ and $0 \leq s < t$, there holds
\begin{align}
    \begin{split}\label{eq55}
        V_n(x_n, s_n,t_n) &\leq V_n(y_n, s_n) +(t_n - s_n) f_n^*\left( \frac{x_n - y_n}{t_n - s_n}\right) \\
        &\leq V_n(z,s_n) + (t_n - s_n) f^*_n \left( \frac{x_n - z}{t_n -s_n}\right).
    \end{split}
\end{align}
The sequence $\{y_n\}$ is bounded as $\left|\frac{x_n - y_n}{t_n - s_n}\right| \leq M$ and therefore, for $y$ a limit point, there is a subsequence $\{y_{n_k}\}$ converging to $y$. The first part of this lemma and the (Lemma \ref{lem2}) tells that
\begin{align}
    \begin{split}
        V(x,s,t) & = \lim_{n \rightarrow \infty}V_{n_k}(x_{n_k}, s_{n_k}, t_{n_k} ) \\
        &= V(y,s) + (t-s) f^*\left(\frac{x-y}{t-s}\right) \\
        &\leq V(z,s) + (t-s) f^*\left(\frac{x-z}{t-s}\right),
    \end{split}
\end{align}
which tells that $y \in Ch(x,s,t)$ and this proves the second part of the lemma.

For the last part of the lemma, observe that $u = \frac{\pa V}{\pa x}$ satisfies \cref{weak}, by the \cref{it3005} of the  \cref{lem2}. Now, fix a function $\varphi \in C_c^{\infty}(\re \times (0,\infty))$. Since $V$ and $V_n$'s are lipshitz continuous functions, the first part of this \cref{lem3} along with integration by parts gives the following integral equalities :
\begin{align}
    \begin{split}
        \int_0^{\infty} \int_{- \infty}^{\infty}u(x,t) \varphi(x,t) dx dt &= \int_0^{\infty} \int_{- \infty}^{\infty}\left(\frac{\pa}{\pa x}V(x,t)\right) \varphi(x,t) dx dt \\
        &= -\int_0^{\infty} \int_{- \infty}^{\infty}V(x,t) \left(\frac{\pa}{\pa x} \varphi(x,t) \right)dx dt\\ 
        &= - \lim_{n \rightarrow\infty} \int_0^{\infty} \int_{- \infty}^{\infty} V_n(x,t)  \left(\frac{\pa}{\pa x} \varphi(x,t) \right)dx dt\\
        & =  \lim_{n \rightarrow \infty} \int_0^{\infty} \int_{- \infty}^{\infty}\left(\frac{\pa}{\pa x}V_n(x,t)\right) \varphi(x,t) dx dt \\
        &= \lim_{n \rightarrow \infty} \int_0^{\infty} \int_{- \infty}^{\infty}u_n(x,t) \varphi(x,t) dx dt,
    \end{split}
\end{align}
which concludes the third point of the lemma.
\end{proof}

Now, we state the Lax-{O}le\u{\i}nik approach for explicit formula and the one sided inequality. The proof can be found in \cite{Evans}.

\begin{Lemma}
Assume that the function $f : \re \mapsto \re$ is uniformly convex with $f''(\theta) \geq C >0$, for all $\theta$ in $\re$. For $u_0 \in L^{\infty}(\re)$, let $v_0$ be the primitive of $u_0$ and $V$ be the associated value function as in (\ref{eq8}) and (\ref{eq13}).
The function $u := \frac{\pa V}{\pa x}$ is a weak solution to the PDE (\ref{SC1}) and for $ t > 0$, the function $y(x,t) = y_+(x,t)$, defined in (\ref{eq1011}) satisfy,
\begin{itemize}
    \item The mapping $x \mapsto y(x,t)$ is a non-decreasing function.
    \item For a.e $x \in \re$, there holds the equality
    \begin{equation}
        u(x,t) = \left(f^*\right)' \left( \frac{x - y(x,t)}{t}\right)
    \end{equation}
\end{itemize}
Furthermore, the function $u$ satify the {O}le\u{\i}nik-one-sided inequlaity mentioned in (\ref{OOS}) i.e.
\[ u(x+z,t) - u(x,t) \leq C( 1 +t^{-1})z\]
\end{Lemma}

\begin{Remark}
Here, since $f$ is uniformly convex, we have $(f^*)' = (f')^{-1}$.
\end{Remark}

\newpage
\section*{\bf Proof of the Main Theorems.}
First, let's recall some known results whose proofs can be found in \cite{Evans}.

Assume that the function $f : \re \mapsto \re$ is uniformly convex with $f''(\theta) \geq C > 0$, for all $\theta \in \re$. Choose a non-negative function $\rho \in C_c^{\infty}(\re^2)$ such that 
\begin{itemize}
    \item The support of the function $\rho$ satisfies
    \[ supp(\rho) \subset \{ (x,t) \in \re^2 ; t \leq 0\}.\]
    \item The intgeral of $\rho$ is 1, i.e
    \begin{align}\label{eq59}
        \int_{\re^2} \rho(x,t) dx dt = 1.
    \end{align}
\end{itemize}
For $\epsilon >0$, let \[\left\{ \rho_{\epsilon}(x,t) := \frac{1}{\epsilon ^2} \rho \left(\frac{x}{\epsilon}, \frac{t}{\epsilon}\right)\right\}, \] be the mollifying sequence for the function $\rho$ and for $h \in L^{\infty}(\re)$, set 
\begin{align}
    h_{\epsilon}(x,t) := \left(\rho_{\epsilon} \ast h \right)(x,t).
\end{align}
Then, the function $h_{\epsilon} \in C^{\infty}(\re^2)$ and there holds the inequality,
\begin{align}
    \lVert h_{\epsilon} \rVert_{\infty} \leq \lVert h \rVert_{\infty}.
\end{align}
Suppose there exist $C_1 >0$ such that for all $t >0$, for a.e $x \in \re$, $z > 0$, the function $h$ satisfy
\begin{align}\label{eq62}
    \frac{h(x+z,t)-h(x,t)}{z} \leq \frac{C_1}{t},
\end{align}
Then, we have
\begin{equation}\label{eq63}
    \begin{split}
        \frac{\pa}{\pa x}& h_{\epsilon}(x,t) = \lim_{z \rightarrow 0^+}\frac{h_{\epsilon}(x+z,t) - h_{\epsilon}(x,t)}{z}\\
        &= \lim_{z \rightarrow 0^+} \int_{\tau = -\infty}^0 \int_{y= -\infty}^{\infty} \left( \frac{h(x-y+z, t- \tau) - h(x-y,t-\tau)}{z}\right)\rho_{\epsilon}(y,\tau) dy d\tau \\
        & \leq \int_{\tau = -\infty}^0 \int_{y= -\infty}^{\infty}  \left(\frac{C_1}{t - \tau}\right) \rho_{\epsilon}(y,\tau) dy d\tau \\
        & \leq \frac{C_1}{t}.
    \end{split}
\end{equation}

For $h_1, h_2 \in L^{\infty}(\re^2)$, define the quantities, 
\begin{align}
    \begin{split}\label{eq64}
        H(x,t) &:= \frac{f(h_1(x,t)) - f(h_2(x,t))}{h_1(x,t) - h_2(x,t)} \\
        & = \int_0^1 f'\Big[ \lambda h_1(x,t) + (1-\lambda) h_2(x,t) \Big] d\lambda,
    \end{split}
\end{align}
\begin{align}
    \begin{split}\label{eq65}
        H_{\epsilon}(x,t) := \int_0^1 f'\Big[ \lambda h_{1\epsilon}(x,t) + (1-\lambda) h_{2\epsilon}(x,t) \Big]
        d\lambda,
    \end{split}
\end{align}

\begin{align}
    \begin{split}\label{eq66}
    M &:= \max_{\lambda \in [0,1]} \lVert \lambda h_1 + (1-\lambda)h_2 \rVert_{\infty},\\
    L &:= \max_{\theta \in [-M,M]}|f'(\theta)|, \\
    L_1 &:= \max_{\theta \in [-M,M]}|f''(\theta)|.
    \end{split}
\end{align}
The notations (\ref{eq64}) - (\ref{eq66}) yield the following conclusions:
\begin{enumerate}
    \item\label{it321} The value $\max\left(\lVert H\rVert_{\infty}, \lVert H_{\epsilon}\rVert_{\infty}\right)$ is less than or equal to $L$.
    \item The function $H_{\epsilon}$ is in the space $C^{1}(\re^2)$.
    \item The functions $H_{\epsilon}$ converges to $H$ in $L^1_{loc}(\re^2)$ as $\epsilon$ goes to $0$.
    \item Suppose that $h_1, h_2$ satisfy (\ref{eq62}), then from (\ref{eq63}), for $t>0$, we have
    \begin{equation}\label{eq67}
        \frac{\pa}{\pa x}\Big[ \lambda h_{1\epsilon}(x,t) + (1-\lambda) h_{2\epsilon}(x,t) \Big] \leq \frac{C_1}{t}.
    \end{equation}
    \item As $f''$ is assumed to be positive, the relation (\ref{eq65}) tells that for $t>0$, we have
    \begin{align}
        \begin{split}\label{eq68}
            \frac{\pa H_{\epsilon}}{\pa x}(x,t) &= \int_0^1 f''(\lambda h_{1,\epsilon}+ (1-\lambda) h_{2,\epsilon}) \frac{\pa}{\pa x} \left[\lambda h_{1\epsilon} + (1-\lambda) h_{2,\epsilon}\right] d\lambda \\
            &\leq \frac{C_1}{t}\int_0^1 f''(\lambda h_{1,\epsilon}+ (1-\lambda) h_{2,\epsilon}) d\lambda \\
            &\leq \frac{C_1 L_1}{t}.
        \end{split}
    \end{align}
\end{enumerate}
Assuming the properties mentioned in (\ref{eq59}) - (\ref{eq68}), we have the following lemma.
\begin{Lemma}\label{lem5}
 Let $f$ be a uniformly convex function with $u_{10}, u_{20} \in L^{\infty}(\re)$ and let $u_1$ and $u_2$ be two weak solutions to the PDE (\ref{SC1}) satisfying the {O}le\u{\i}nik-one-sided inequality (\ref{OOS}). Then, for $a < b$, $0 < \tau < T$, $\psi \in C_c^{\infty}((a,b)\times (\tau,T))$, there holds that 
\begin{equation}
    \left| \int_0^{\infty} \int_{- \infty}^{\infty} (u_1 - u_2)\psi \ dx dt \right| \leq \lVert \psi \rVert_{\infty} (T-\tau) \int_{a-LT}^{b+LT} |u_{10}(x) - u_{20}(x)| dx.
\end{equation}
\end{Lemma}

\begin{proof}
Setting $u(x,t) \equiv 0$ for $t<0$, we can assume that the functions $u_i \in L^{\infty}(\re^2)$, for $i = 1,2$. For $i=1,2$, define $h_i$ to be $ u_i$ and $h_{i\epsilon}$ to be $ u_{i\epsilon}$. Furthermore, set $H$, $H_{\epsilon}$ to be the functions as in (\ref{eq64}) and (\ref{eq65}). Now, for $(x,t) \in \re^2$, define the following:
\begin{itemize}
    \item $w_0(x) := u_{10}(x) - u_{20}(x)$,
    \item $w(x,t) := u_1(x,t) - u_2(x,t)$,
    \item A function $\chi(\theta) \equiv \chi(\theta,x,t)$ which solves the ODE :
    \begin{equation}\label{eq70}
    \left\{ \begin{aligned}
        \frac{d \chi}{d \theta}(\theta) &= H_{\epsilon}(\chi(\theta,x,t),\theta) \\
        \chi(t,x,t) &= x
        \end{aligned}\right.
        \end{equation}
    \item The function $\veps$ which is a solution to
    \begin{equation}\label{eq71}
        \left\{
        \begin{aligned}
            \left( \frac{\pa \veps}{\pa t} + H_{\epsilon} \frac{\pa \veps}{\pa x}\right)(x,t) &= \psi(x,t), \quad &t<T, x \in \re, \\
            \veps(x,T) &= 0, \quad &\forall x \in \re.
        \end{aligned}\right.
    \end{equation}
    i.e the function $\veps$ is given by
    \begin{equation}\label{eq72}
        \veps(x,t) = - \int_t^T \psi(\chi(\theta,x,t),\theta) d\theta.
    \end{equation}
    \item View $H(x,t)$ and $H_{\epsilon}(x,t)$ as functions of the form $H(\xi, t)$ and $H_{\epsilon}(\xi,t)$. The \cref{eq70} gives
    \begin{equation}\label{eq73}
        \left\{\begin{aligned}
            \frac{d}{d \theta}\left(\frac{\pa \chi}{\pa x}\right) &= \frac{\pa H_{\epsilon}}{\pa \xi}\left(\chi(\theta,x,t),\theta\right) \frac{\pa \chi}{\pa x}(\theta,x,t), \\
            \frac{\pa \chi}{\pa x}(t,x,t) &= 1,
        \end{aligned}\right.
    \end{equation}
    which tells that
    \begin{equation}\label{eq74}
        \frac{\pa \chi}{\pa x}(\theta,x,t) = \int_t^{\theta} exp\left(\int_t^s \frac{\pa H_{\epsilon}}{\pa \xi}\left(\chi(\alpha,x,t),\alpha\right)d\alpha\right) ds.
    \end{equation}
\end{itemize}
Thus, the function $\frac{\pa \chi}{\pa x}$ is non negative and from (\ref{eq68}), along with $C_2 = C_1 L_1$, we have
\begin{equation}\label{eq75}
\frac{\pa \chi}{\pa x}(\theta, x ,t) \leq \int_t^{\theta}exp\left(C_2 \log\left(\frac{s}{t}\right)\right) ds \leq \frac{\theta^{C_2 +1}}{C_2 t^{C_2}}.
\end{equation}
Hence, for $0<t<T$ and for $C_3 := \frac{T^{C_2 + 2}}{C_2 (C_2 +1)}$, we see that 
\begin{align}
    \begin{split}\label{eq76}
        \left|\frac{\pa \veps}{\pa x}(x,t)\right| &= \left| \int_t^T \frac{\pa \psi}{\pa \xi}(\chi(\theta,x,t),\theta) \frac{\pa \chi}{\pa x}(\theta,x,t) d\theta \right| \\
        &\leq \frac{C_3}{t^{C_2}}\left|\left| \frac{\pa \psi}{\pa \xi}\right|\right|_{\infty}.
    \end{split}
\end{align}
As $\supp\psi$ is contained in $\{(x,t) ; t > \tau\}$, for $0<t<\tau, x \in \re, t < \theta<\tau$, we have
\begin{equation}\label{eq77}
\begin{split}
    \frac{d}{d \theta} \varphi(\chi(\theta,x,t),\theta) &= \left( \frac{\pa \veps}{\pa t} + H_{\epsilon}\frac{\pa \veps}{\pa x}\right) \left(\chi(\theta,x,t),\theta\right) \\
    &= \psi(\chi(\theta,x,t),\theta) \\
    &= 0.
    \end{split}
\end{equation}
This tells that $\veps(x,t) = \veps(\chi(\tau,x,t),\tau)$. Thus, the mean value theorem implies for $0<t<\tau$,
\begin{equation}\label{eq78}
    \int_{-\infty}^{\infty} \left| \frac{\pa \veps}{\pa x}(x,t)\right|dx \leq \int_{-\infty}^{\infty}\left| \frac{\pa \veps}{\pa x}(x,\tau)\right| dx.
\end{equation}
Now, since $u_1$ and $u_2$ are weak solutions, for $0<\tau_1<\tau<T$, there holds
\begin{equation}\label{eq79}
\begin{split}
\int_{-\infty}^{\infty}\int_0^{\infty} w \psi dx dt &= \int_{-\infty}^{\infty}\int_0^{\infty} w\left(\frac{\pa \veps}{\pa t} + H_{\epsilon}\frac{\pa \veps}{\pa x}\right) dx dt \\ 
&:= -I_1 + I_2 + I_3,
\end{split}
\end{equation}
where, the terms $I_j$'s are given by
\begin{equation}
\begin{split}
I_1 &:=\int_{-\infty}^{\infty}w_0 (x) \veps(x,0) dx ,\\ 
I_2 &:=  \int_{0}^{\tau_1}\int_{-\infty}^{\infty} \left(H_{\epsilon} - H\right) \frac{\pa \veps}{\pa x}w dx dt, \\
I_3 &:= \int_{\tau_1}^T\int_{-\infty}^{\infty}\left(H_{\epsilon}-H\right) \frac{\pa \veps}{\pa x}w dx dt.
    \end{split}
\end{equation}
\underline{Estimation of $I_1,I_2,I_3$ :}

From \cref{eq70}, for $x \in \re$, $0\leq t <\theta<T$, we have
\begin{equation}\label{equation85}
    \chi(\theta,x,t) = \chi(t,x,t) + \int_t^{\theta} H_{\epsilon}(\chi(s,x,t),s) ds.    
\end{equation}
For $0<\theta<T$, the \cref{eq41} and \cref{equation85} tells
\begin{equation}
    |\chi(\theta,x,0) - x| \leq \lVert H_{\epsilon} \rVert_{\infty}T \leq LT.
\end{equation}
Thus, there holds
\begin{equation}\label{eq82}
    x - LT \leq \chi(\theta,x,t) \leq x + LT.
\end{equation}
Therefore, if $x+LT \leq a \iff x \leq a-LT$, then we have $\chi(\theta,x,t) \leq a$. If $x-LT \geq b \iff x \geq b+LT$, then we have $\chi(\theta,x,t) \geq b$.

Thus, the \cref{eq72} tells $\veps(x,t) = 0$ for $x\notin [a-LT,b+LT]$ and  
\begin{equation}\label{eq83}
    |\veps(x,t)| \leq \lVert \psi\rVert_{\infty}(T - \tau).
\end{equation}
Therefore, $I_1$ can be estimated as
\begin{equation}\label{eq84}
    \begin{split}
        |I_1| & = \left| \int_{-\infty}^{\infty} w_0(x) \veps(x,0) dx \right| \\
        &\leq \lVert \psi\rVert_{\infty}(T-\tau) \int_{a-LT}^{b+LT}|w_0(x)| dx.
    \end{split}
\end{equation}

For the part of $I_2$, for $0< \tau_1 < \tau < T$, the conclusion (\ref{it321}) and (\ref{eq82}) yields
\begin{equation}\label{eq85}
    \begin{split}
        |I_2| &= \left| \int_0^{\tau_1} \int_{-\infty}^{\infty} (H_{\epsilon}-H) \frac{\pa \veps}{\pa x} w dx dt \right| \\
        &\leq \underbrace{2 \lVert w \rVert_{\infty} L \tau_1 \int_{-\infty}^{\infty} \left|\frac{\pa \veps}{\pa x}(x,\tau)\right| dx.}_{\text{goes to $0$ as $\tau_1$ goes to $0$.}}
    \end{split}
\end{equation}
Lastly, the estimation on $I_3$ can be done in the following way. The conclusion (\ref{it321}) gives the convergence of $H_{\epsilon}$ to $H$ in $L^1_{loc}$. Therefore, (\ref{eq76}) and (\ref{eq83}) gives
\begin{equation}
    \begin{split}
        |I_3| &= \left| \int_{\tau_1}^T \int_{-\infty}^{\infty} (H_{\epsilon} - H) \frac{\pa \veps}{\pa x} w dx dt \right| \\
        &\leq \int_{\tau_1}^T\int_{a-LT}^{b+LT} |H_{\epsilon} - H| \left| \frac{\pa \veps}{\pa x}\right| |w| dx dt \\
        &\leq \underbrace{\frac{C_3 \lVert\frac{\pa \psi}{\pa \xi}\rVert_{\infty}}{\tau_1^{C_2}}\lVert w\rVert_{\infty}\int_{\tau}^T\int_{a-LT}^{b+LT} |H_{\epsilon} - H| dx dt.}_{\text{goes to $0$ as $\epsilon$ goes to $0$.}} 
    \end{split}
\end{equation}
Sending $\epsilon$ to $0$ and then $\tau_1$ to $0$, tells $|I_2| + |I_3| \rightarrow 0.$ Thus, by (\ref{eq79}) and (\ref{eq84}), we see that
\begin{equation}\label{equar95}
    \left| \int_{-\infty}^{\infty}\int_0^T w \psi dx dt\right| \leq\lVert \psi\rVert_{\infty}(T-\tau) \int_{a - LT}^{b+LT}|w_0(x)| dx,
\end{equation}
which proves the lemma.

\end{proof}

Using the above results, we now prove the 
Theorem (\ref{thm1}), the Theorem (\ref{thm2}) and the Theorem (\ref{thm3}).

\begin{proof}[\bf{Proof of the Theorem (\ref{thm1})}]
The first four parts of the Theorem (\ref{thm1}) follows from the Lemma (\ref{lem2}). To conclude the theorem, we have to prove the last part of it. Define $f_{\eta}$ as in (\ref{mol}), with renaming $\epsilon$ to be $\eta$, set
\[f_{\eta}(p) := (f\ast \alpha_{\eta})(p) + \eta p^2.\]
Then, $\{f_{\eta}\}$ is uniformly convex, smooth and converges to $f$ on compact subsets by the \cref{lem1}. Since, $f$ is convex, $f'$ exists almost everywhere and the Dominated Convergence Theorem gives
\begin{equation}\label{eq89}
    f_{\eta}'(p) = (f' \ast \alpha_{\eta})(p) + 2 \eta p.
\end{equation}

Set 
\begin{equation*}
    \begin{split}
    M &:= \max \{\lVert u_{10}\rVert_{\infty},\lVert u_{20}\rVert_{\infty}\} \\
    I_{\eta} &:= [-M-\eta , M + \eta] \\ 
    L &:= \limsup_{\eta \rightarrow 0} \{ |f'(q)| ; q \in I_{\eta}\}.
    \end{split}
\end{equation*} 

Then, for $|p| \leq M$, we have
\begin{equation}
    \begin{split}
        |f_{\eta}'(p)| &\leq |(f'\ast \alpha_{\eta})(p)| + 2 \eta M \\
        &\leq \sup \{ |f'(q)| ; q \in I_{\eta}\} + 2 \eta M,
    \end{split}
\end{equation}
and hence, we see that
\begin{equation}\label{eq91}
    \begin{split}
        \lim_{\eta \rightarrow 0} |f_{\eta}'(p)| \leq L.
    \end{split}
\end{equation}
 For $i=1,2$, define $u_{i\eta}$ to be the weak solution to the PDE :
\begin{equation}\label{92}
\left\{ \begin{aligned}
    \frac{\pa}{\pa t}u_{i\eta} + \frac{\pa}{\pa x}\big[f_{i\eta}(u_{i\eta})\big] &= 0 ; \quad x \in \re, t>0, \\
    u_{i\eta}(x,0) &= u_{i0}(x) ; \quad x \in \re,
\end{aligned} \right.
\end{equation}
and set $\omega_{\eta} := u_{1\eta} - u_{2\eta}$. Also, set $w_0 := u_{10} - u_{20}$.

For $0 < \tau < T$, $a<b$, $\psi \in C_c^{\infty}((a,b)\times(\tau,T))$, the Lemma (\ref{lem5}) gives 
\begin{equation}
    \left| \int_{-\infty}^{\infty} \int_0^{\infty} \omega_{\eta} \psi dx dt \right| \leq\lVert \psi\rVert_{\infty}(T-\tau)\int_{a - LT}^{b+LT}|\omega_0(x)|dx.
\end{equation}

On the compact set $\supp(\psi)$, \cref{lem1} tells that $f_{\eta}$ converges to $f$. Now, from the Stability Lemma (\ref{lem3}), $\omega_{\eta}$ converges to $\omega$ in $\mathscrsfs{D}'(\re \times(0,\infty))$ as $\eta$ goes to $0$. Thus, sending $\eta$ to $0$, we see that
\begin{equation}
    \left|\int_a^b \int_{\tau}^T \omega \psi dx dt \right| \leq\lVert \psi\rVert_{\infty}(T- \tau) \int_{a - LT}^{b+LT}|\omega_0(x)| dx.
\end{equation}
Now, letting $\psi \rightarrow \frac{\omega}{|\omega|}$ gives
\begin{equation}
    \int_a^b \int_{\tau}^T |\omega|  dx dt \leq (T- \tau) \int_{a - LT}^{b+LT}|\omega_0(x)| dx.
\end{equation}
Thus for a.e $T>0$, by the Lebesgue differentiation theorem (refer \cite{Stein}), we have
\begin{equation}
    \lim_{\tau \rightarrow T} 
    \frac{1}{T - \tau}\int_{\tau}^T\left(\int_a^b | \omega(x,\theta)| dx \right)d\theta = \int_a^b |\omega(x,T)|dx,
\end{equation}
which proves the fifth point of the first theorem i.e.
 \begin{equation}\label{infinite} 
 \int_a^b |\omega(x,T)|dx \leq \int_{a - LT}^{b+LT}|\omega_0(x)| dx. 
 \end{equation}

\end{proof}

\begin{proof}[\bf{Proof of the Theorem (\ref{thm2}).}]
Set $\omega_0 := u_{10} - u_{20}$ which is non positive function and let $0 < \tau<T$, $\psi \in C_0^1(\re\times(0,T))$ be a function such that
\begin{equation}
    \psi(x,t)  \geq 0 , \quad \forall (x,t) \in \re\times(0,\infty).
\end{equation}
Let $\veps$ be as in (\ref{eq71}) which tells $\veps(x,0) \leq 0$ for all $x \in \re$ by (\ref{eq72}) and by assumption, $\omega_0 \leq 0.$ Now, from (\ref{eq72}) and (\ref{eq79}), we have
\begin{equation}\label{eq102}
    \begin{split}
        &\quad \int_{-\infty}^{\infty} \int_0^{\infty} (u_1(x,t) - u_2(x,t)) \psi(x,t) dx dt \\
        &= - \int_{-\infty}^{\infty} \omega_0(x) \veps(x,0) dx  + \int_0^{T} \int_{-\infty}^{\infty}  (H_{\epsilon} - H) \frac{\pa \veps}{\pa x} \omega dx dt \\
        &\leq \int_0^{T} \int_{-\infty}^{\infty} (H_{\epsilon} - H) \frac{\pa \veps}{\pa x} \omega dx dt \quad \underset{ \epsilon\rightarrow 0}{\longrightarrow} 0.
    \end{split}
\end{equation}
Thus, for all positive $\psi \in C_0^1((a,b) \times (\tau,T))$, we see that
\begin{equation}
    \int_a^b \int_{\tau}^T \big[ u_1(x,t) -u_2(x,t) \big]\psi(x,t) dx dt \leq 0,
\end{equation}
which tells that for a.e $(x,t) \in \re\times(0,\infty)$, the functional inequality 
\begin{equation}
    u_1(x,t) \leq u_2(x,t),
\end{equation}
which proves the first part of the theorem.

For $ a = - \infty, b = +\infty$, for $i\in\{1,2\}$, we have
\begin{equation}
\begin{split}
    \int_{-\infty}^{\infty} \int_{\tau}^T |u_{in}(x,t) &- u_{im}(x,t)| dx dt \\ &\leq (T- \tau) \underbrace{\int_{-\infty}^{\infty}|u_{0in}(x)-u_{0im}(x)|dx,}_{\text{ goes to } 0 \text{ as } n,m \rightarrow \infty, \text{ by hypothesis.}}
\end{split}
\end{equation}
which tells that $\{u_{in}\}$ is a cauchy sequence in $L^1(\re\times(\tau,T))$. Thus, there exists $u_i \in L^1(\re \times (0,T))$ such that $\lim_{n \rightarrow \infty} u_{in}(x,t) = u_i(x,t)$. The $L^1$ contraction property then gives
\begin{equation}\label{eq106}
    \begin{split}
        \int_{-\infty}^{\infty} \int_{\tau}^T |u_1(x,t) &- u_2(x,t)| dx dt \\ &= \lim_{n \rightarrow\infty} \int_{-\infty}^{\infty} \int_{\tau}^T |u_{1n}(x,t) - u_{2n}(x,t)| dx dt \\
        &\leq (T-\tau) \lim_{n \rightarrow\infty}\int_{-\infty}^{\infty}|u_{10n}(x) - u_{20n}(x)| dx \\
        &= (T-\tau) \int_{-\infty}^{\infty}|u_{10}(x) - u_{20}(x)|dx,
    \end{split}
\end{equation}
which proves \cref{eq319} and taking $u_{10} = u_{20}$ in \cref{eq106} gives \cref{eq320}. Finally, we have 
\begin{equation*}
\begin{split}
    \lim_{\tau \rightarrow T}  
    \frac{1}{T - \tau}\int_{\tau}^T\Big(\int_{\re} | u_1(x,t) - u_2(x,t)&| dx \Big)dt =
    \\ &\int_{\re} |u_{1}(x,T) - u_{2}(x,T)|dx.
    \end{split}
\end{equation*}

So, for a.e $T>0$, we have
\[ 
 \int_{\re} |u_{1}(x,T) - u_{2}(x,T)| dx \leq \int_{-\infty}^{\infty}|u_{10}(x) - u_{20}(x)|dx.
\]

From the last part of this theorem, we have for $u_0 \in L^1(\re)$, the function constructed $u$ is in $L^1(\re\times (0,T))$. But, it is not yet clear if $f(u(x,t))$ is well defined and satisfy the equation (\ref{SC1}). We shall prove this in several steps.

Let $u_0 \in L^1(\re)$ and $u_{0n} \in L^{\infty}(\re) \cap L^1(\re)$ such that $u_{0n}$ converges to $u_0$ in $L^1(\re)$ as $n$ goes to infinity. Let $\epsilon >0$ and $\fe$ be as in (\ref{mol}). Let $V_{\epsilon,n}$ be the value function as in (\ref{eq13}) with the flux $\fe$  and the initial data $u_{0n}$. Furethermore, let the corresponding charecteristic set be $Ch_{\epsilon,n}(x,t)$ with $y_{\pm ,\epsilon,n}(x,t)$ as defined in (\ref{eq13}), (\ref{eq1011}). Let $K$ be a compact subset of $\re \times (0,\infty)$. Then, the following holds.

\begin{enumerate}[Step 1:]
\item\label{step1} There exist $C \equiv C(K) >0 $, independent of $\epsilon$ and $n$ such that for any $(x,t) \in K, n>0, y \in Ch_{\epsilon,n}(x,t)$, there holds

\begin{equation}
    |y| \leq C(K).
\end{equation}
\begin{proof}
    Suppose not, then there is a sequence $\epsilon_k \rightarrow 0$, $(x_k,t_k) \in K$, $n_k \rightarrow \infty$, $y_k \in Ch_{\epsilon_k,n_k}(x_k,t_k)$ such that,
    \begin{itemize}
        \item $\lim_{k \rightarrow \infty} (x_k,t_k) = (x_0,t_0) \in K$.
        \item $\lim_{k \rightarrow \infty} |y_k| = \infty$.
    \end{itemize}
    Since, $y_k \in Ch_{\epsilon_k,n_k}(x_k, t_k)$, we see that for all $z \in \re$, there holds
    \begin{equation*}
    \begin{split}
    V_{\epsilon_k,n_k } (x_k ,t_k) & =  v_{0 n_k}(y_k) + t_k f^*_{\epsilon_k} \left( \frac{x_k - y_k}{t_k}\right) \\ & \leq   v_{0 n_k}(z) + t_k f^*_{\epsilon_k}\left( \frac{x_k - z}{t_k}\right), 
    \end{split}
    \end{equation*}
    where, 
    \[ v_{0,n_k}(z) := \int_0^z u_{0, n_k}(\theta) d\theta. \]
    From the convergence of $u_{0,n}$ to $u_0$ in $L^1$, we see that there here exists $k_0 \geq 1$ such that for all $ z$ and for all $k \geq k_0$, we have 
    \[|v_{0,n_k}(z)| \leq \int_{- \infty}^{\infty} |u_{0,n_k}(\theta)| d\theta \leq 2 \int_{- \infty}^{\infty} |u_0(\theta)| d \theta .\]
    Also, note that compact set $K$ lies strictly in the upper half plane, which tells that for all $(x,t) \in K$, the time factor $t$ is strictly bigger than some positive number.  Now, for $k \geq k_0$, evaluating at $z =0$, we have
    \[ t_k f^*_{\epsilon_k}\left( \frac{x_k - y_k}{t_k}\right) \leq 
 2 \lVert u_0\rVert_{L^1} + t_k f^*_{\epsilon_k}\left(\frac{x_k}{t_k} \right). \]
    Letting $k$ going to infinity, we see that 
    \begin{equation}\label{equa108} \lim_{k \rightarrow \infty} f^*_{\epsilon_k} \left( \frac{x_k - y_k}{t_k}\right) \leq \frac{2 \lVert u_0 \rVert_{L^1}}{t_0} + f^*\left(\frac{x_0}{t_0}\right).
    \end{equation}
    Now, $f_\epsilon$ goes to $f$ on compact sets tells that $f_\epsilon$ is uniformly bounded on $[-1,1]$. So, by the definition of the Fenchel dual, there exist $q_0 \geq 1$ such that for $|q| \geq q_0$ and for the particular $p = \frac{q}{|q|}$, there holds
    \[\frac{\fe^*(q)}{|q|} \geq 1 - \frac{\fe\left(\frac{q}{|q|}\right)}{|q|} \geq \frac{1}{2},\]
    or equivalently, there holds
    \begin{equation}\label{equa109}
        \fe^*(q) \geq \frac{1}{2} |q|.
    \end{equation}
Now, $|y_k|$ goes to infinity implies that for $k$ large, we have $\left|\frac{x_k - y_k}{t_k}\right| \geq q$. Along with (\ref{equa108}) and (\ref{equa109}), we have
\begin{equation}
    \begin{split}
    \infty &= \frac{1}{2}\lim_{k \rightarrow \infty} \left| \frac{x_k - y_k}{t_k}\right| \\ &\leq \lim_{k \rightarrow \infty} f^*_{\epsilon_k} \left( \frac{x_k - y_k}{t_k}\right) \\ &\leq \frac{2 \lVert u_0 \rVert_{L^1}}{t_0} + f^*\left(\frac{x_0}{t_0} \right), 
\end{split}
\end{equation}
which is a contradiction.
\end{proof}
\item\label{step2} We have the limit,
\begin{equation}
    \lim_{|q| \rightarrow \infty} \inf_{0<\epsilon<1} \left| f'_{\epsilon}(q) \right| = \infty.
\end{equation}
\begin{proof}
    Since, $f$ has super-linear growth and convex, we see that \begin{equation}\label{equa110} \lim_{|q| \rightarrow \infty} \left| f'(q)\right| = \infty.\end{equation}
By the Dominated Convergence Theorem, we have
\[f'_{\epsilon}(q) = \int_{|y| \leq 1} f'(q - \epsilon y) \alpha(y) dy + 2 \epsilon q. \]
If $q \rightarrow \infty $, by (\ref{equa110}), we see that
\[ \lim_{q \rightarrow \infty} \inf_{  0 < \epsilon <1, |y| \leq 1 } f'(q - \epsilon y) = \infty. \]
Thus, by the Fatou's lemma, there holds
\[\lim_{q \rightarrow \infty} \inf_{0 < \epsilon < 1} f'_{\epsilon}(q) \geq \int_{|y| \leq 1} \left( \liminf_{q \rightarrow \infty} f'(q - \epsilon y)\right) \alpha(y) dy = \infty. \]
Similarly, if $q \rightarrow - \infty$, then from (\ref{equa110}), we have 
\[\lim_{q \rightarrow -\infty} \inf_{0 < \epsilon <1, |y| \leq 1} \left( - f'(q - \epsilon y)\right) = \infty.\]
    and from convexity of $f$ i.e $f'$ is decreasing near $-\infty$, we see that 
\begin{equation}
\lim_{q \rightarrow -\infty} \inf_{0 < \epsilon < 1} \left( - f'_{\epsilon}(q) \right) \geq \int_{|y| \leq 1} \left( \liminf_{q \rightarrow- \infty} -f'(q - 1) \right) \alpha(y) dy = \infty. 
\end{equation} 

\end{proof}
\item\label{step3} Let $u_{\epsilon,n}$ be the solution of the PDE (\ref{SC1}) with the flux $\fe$ and the initial data $u_{0n}$. Then, by the Lax-{O}le\u{\i}nik explicit formula, for $t >0$ and a.e $x \in \re$, we see there exist $y_{+,\epsilon,n}$ (as defined in \cref{eq1011}) such that
\begin{equation}
    f'_{\epsilon} (u_{\epsilon,n}(x,t)) = \frac{x - y_{+,\epsilon,n}(x,t)}{t}
\end{equation}
Let $K \subset \re \times (0,\infty)$ be a compact set. Then, from the (\ref{step1}), there exist $C(K) > 0$ such that for all $0 < \epsilon <1$, $(x,t) \in K$, for all $n$, we have
\begin{equation}
    |\fe ' (u_{\epsilon,n}(x,t))| = \left| \frac{x - y_{+,\epsilon,n}(x,t)}{t}\right| \leq C(K).
\end{equation}
From \cref{3oflem3} of the \cref{lem3}, letting $\epsilon \rightarrow 0$, we obtain the limit $u_{\epsilon,n}(x,t) \rightarrow u_n(x,t)$ in $\mathcal{D}'(\re\times(0,\infty))$ and from \cref{equar95}, $u_n(x,t)$ is in $L^1_{loc}(\re\times(0,\infty))$. From the (\ref{step2}), it is seen that the set $\{ u_{\epsilon,n}(x,t)\}$ is uniformly bounded, for all $(x,t) \in K$, for a fixed $n \in \mathbb{N}$ and for all $\epsilon$ near zero. 

Now, we show that the uniform bound can be taken to be independent of $n$ as well. Let $K \subset \re \times(0,\infty)$ be a rectangle and $\Omega := int(K)$, the interior of the set $K$. Set the terms in the \cref{prop_new} mentioned in the Appendix, as 
\begin{itemize}
    \item $w_k(x,t) \equiv u_{\epsilon,n}(x,t)$,
    \item $w(x,t) \equiv u_n(x,t)$.
\end{itemize}
Since, the function $u_n(x,t)$ is defined as $\frac{\pa V_n}{\pa x}(x,t)$, we have
\[
\lVert u_n\rVert_{L^{\infty}(\re\times(0,\infty))} \leq  Lip(V_n) \lVert u_{0,n}\rVert_{L^{\infty}(\re)}.
\]
So, from the \cref{prop_new}, for all $n \in \mathbb{N}$, we see that 
\begin{equation}\label{u_nconverge}
\lVert u_n\rVert_{L^{\infty}(K)} \leq \sup_{k} \lVert u_{\epsilon,n} \rVert_{L^{\infty}(K)}.
\end{equation}
 
Now, the $L^1-$contractivity tells that the functions $u_n(x,t)$ is cauchy in $L_{loc}^1(\re \times (0,\infty))$ and hence, converges to some function $u(x,t)$ in $L_{loc}^1(\re \times (0,\infty))$. The \cref{u_nconverge} tells that the solution $u:= \lim u_n$ is in $L^{\infty}(K)$. The function $f$ is convex and so is continuous. The fact that the $L^1$ convergence implies there exist a subsequence that converge pointwise almost everywhere tells that there is some subsequence  such that $f(u_{n_k}(x,t))$ converges to $f(u(x,t))$, for a.e $(x,t) \in \re \times \re^+$. The $\{u_n(x,t)\}$ is bounded on $K$ tells that by the dominated convergence theorem, for all $\varphi \in C_c^{\infty}(K)$, we have
\begin{equation}\label{equation113}
\begin{split}
    \int_{\re} \int_0^{\infty} \Big[ u \varphi_t + f(u) \varphi_x \Big] dx dt &= \lim_{k \rightarrow \infty} \int_{\re} \int_0^{\infty} \Big[ u_{n_k} \varphi_t + f(u_{n_k}) \varphi_x\Big] dx dt \\
    & = 0.
    \end{split}
\end{equation}

For the last part of the theorem, fix $n\in \mathbb{N}$ and for $\eta >0$, $\epsilon>0$ and $T>0$, define the function $\varphi(x,t) := A_{\epsilon}(x)B_{\eta}(t)$ by, 
\begin{equation}
A_{\epsilon}(x) := \left\{ \begin{aligned}
    1 \quad &\text{ if } x \in [a,b], \\
    0 \quad &\text{ if } x \notin [a - \epsilon,b+\epsilon],\\
    \frac{x-a+\epsilon}{\epsilon} \quad &\text{ if } x \in [a-\epsilon,a],\\
    \frac{b+\epsilon-x}{\epsilon} \quad &\text{ if } x \in [b,b+\epsilon].
\end{aligned} \right.
\end{equation}
\begin{equation}
    B_{\eta}(t) := \left\{
    \begin{aligned}
        1 \quad &\text{ if } t \in [0,T], \\
        \frac{T+\eta-t}{\eta} \quad &\text{ if } t \in [T , T+\eta], \\
        0 \quad &\text{ if } t \geq T+\eta.
    \end{aligned}\right.
\end{equation}
The above defined $\varphi$ is liphsitz and has compact support. Now, from the weak formulation \cref{weak}, for the solution $u_{n}$ satisfying the conservation laws with the initial data $u_{0n} \in L^{\infty}(\re)$, we have
\begin{equation*}
\begin{split}
\frac{-1}{\eta} \int_T^{T+\eta}&\int_{-\infty}^{\infty} u_{n}(x,t) A_{\epsilon}(x) dx dt \\ &+ \int_0^{T+\eta} \int_{-\infty}^{\infty} f(u_{n}(x,t)) \left(A_{\epsilon}(x)\right)_{x} B_{\eta}(t) dx dt \\
& \quad \quad \quad \quad \quad + \int_{-\infty}^{\infty} u_{0n}(x) A_{\epsilon}(x) dx = 0.
\end{split}
\end{equation*}

As $u_{0n}$ is in $L^{\infty}(\re)$, we have that the function $u_n(x,t)$ to be in $L^{\infty}(\re \times (0,\infty))$. So, by the Lebesgue differentiation theorem and the dominated convergence theorem, for a.e $t>0$ depending on $A_{\epsilon}$, sending $\eta \rightarrow 0$, we have
\begin{equation*}
\begin{split}
\int_{-\infty}^{\infty} &u_{n}(x,t) A_{\epsilon}(x) dx  \\ &= \int_0^{t} \int_{-\infty}^{\infty} f(u_{n}(x,\tau)) \left(A_{\epsilon}(x)\right)_{x} B_{\eta}(\tau) dx d\tau \\
& \quad \quad \quad \quad \quad + \int_{-\infty}^{\infty} u_{0n}(x) A_{\epsilon}(x) dx.
\end{split}
\end{equation*}
Now, let $t \rightarrow 0$ to get
\[
\lim_{t \rightarrow 0}\int_{-\infty}^{\infty}u_n(x,t) A_{\epsilon}(x) dx = \int_{-\infty}^{\infty} u_{0n}(x) A_{\epsilon}(x) dx
\]
Equivalently, there holds
\begin{equation*}
    \begin{split}
        &\lim_{t \rightarrow 0} \left[ \int_a^b u_n(x,t) dx + \int_{a - \epsilon}^a u_n(x,t) A_{\epsilon}(x) dx + \int_{b}^{b+\epsilon}u_n(x,t) A_{\epsilon}(x) dx \right] \\
        &\quad \quad = \int_a^b u_{0n}(x) dx + \int_{a - \epsilon}^a u_{0n}(x,t) A_{\epsilon}(x) dx + \int_{b}^{b+\epsilon}u_{0n}(x,t) A_{\epsilon}(x) dx
    \end{split}
\end{equation*}
Observe that the chosen $A_{\epsilon}$ has the range $[0,1]$. So, let $\epsilon \rightarrow 0$ to obtain
\begin{equation}\label{shoot122}
\lim_{t \rightarrow 0} \int_a^b u_n(x,t) dx = \int_a^b u_{0n}(x) dx.
\end{equation}

Finally as $u_n \rightarrow u$ in $L^1(\re \times (0,T))$, for all $T>0$ and by the $L^1-$contractive property, for a.e $t>0$, we see that
\begin{equation}
    \begin{split}
        \left| \int_a^b u(x,t) dx - \int_a^b u_0(x) dx \right| &\leq \left| \int_a^b  (u(x,t) - u_n(x,t)) dx \right| \\ & \quad \quad + \left|\int_a^b (u_n(x,t) - u_{0n}(x))  dx \right| \\& \quad \quad \quad + \left| \int_a^b (u_0(x) - u_{0n}(x)) dx \right|\\
        &\leq 2 \int_{-\infty}^{\infty} |u_0(x) - u_{0n}(x) | dx \\ &\quad \quad + \left|\int_a^b (u_n(x,t) - u_{0n}(x)) dx\right|.
    \end{split}
\end{equation}
From \cref{shoot122} and the fact that $u_{0n}$ converge to $u_0$ in the $L^1$ norm, letting $t \rightarrow 0$ and $n \rightarrow \infty$, we have
\[ 
\lim_{t \rightarrow 0} \int_a^b u(x,t) dx = \int_a^b u_0(x) dx. \]
This, together with (\ref{equation113}) gives that $u$ is a ``Kru\v{z}kov" solution and this concludes the proof for the second theorem.
\end{enumerate}

\end{proof}
\begin{proof}[\bf{Proof of the Theorem (\ref{thm3}).}]
Looking at the possibilities for $\mu_{\pm}$, we have four cases:
\begin{enumerate}
    \item\label{case1} $\mu_+ = \infty$ and $\mu_- = -\infty$.
    \item\label{case2} $\mu_+ = \infty$ and $\mu_- > -\infty$.
    \item\label{case3} $\mu_+ < \infty$ and $\mu_- = -\infty$.
    \item\label{case4} $\mu_+ < \infty$ and $\mu_- > -\infty$.
\end{enumerate}
The Theorem \ref{thm2} deals with the case \ref{case1}. So, it is now enough to prove for the case \ref{case2} and a similar analysis follows for the cases \ref{case3} and \ref{case4}.So, assume that $\mu_+ = +\infty$ and $\mu_- > -\infty$. Also for $n \geq 1$, let $p_n \in (-n-1,-n)$ such that the function $f$ is differentiable at $p_n$. Furthermore, let $f_n$ be the mollification of $f$ at $A = p_n$ and $B = \infty$ as mentioned in \cref{equation117} in the appendix. Also, let $u_0$ be a function in $L^1(\re)$ and define
\begin{equation}
       u_{n0}(x) := \left\{ \begin{aligned}
            &u_0(x) &\text{ if } u_0(x) \geq -n \\
            &0 &\text{ } \text{otherwise.}, \\
        \end{aligned}\right.
\end{equation}
Then, from the Theorem \ref{thm1}, there exist a solution $u_n$ of (\ref{weak}) satisfying $\lVert u_n \rVert_{\infty} \leq \lVert u_{n0}\rVert_{\infty} $. Now, from (\cref{theorem2.1}) of the \cref{thm2}, we have $u_n(x,t) \geq -n$ for a.e $(x,t) \in \re \times(0,\infty)$. Hence, for $m>n$, we have
\begin{equation}\label{equation100000}
    f_n\left(u_n\left(x,t\right)\right) = f_m\left(u_n\left(x,t\right)\right),
\end{equation}
which tells that the functions $u_n$ and $u_m$ are solutions for the same flux $f_m$. Thus, by the $L^1-$contractivity, for $0 < \tau < T$, there holds
\begin{equation}\label{equation119}
    \int_{-\infty}^{\infty} \int_{\tau}^{T} |u_n(x,t) - u_m(x,t)| dx dt \leq \underbrace{(T-\tau) \int_{-\infty}^{\infty}|u_{n0}(x) - u_{m0}(x)| dx.}_{\text{goes to }0 \text{ as }m,n \rightarrow\infty}
\end{equation}
Thus, we have that the functions $\{u_n\}$ to be cauchy in $L^1_{loc}(\re\times(0,\infty))$. Now, as $u_{n+1,0}(x) \leq u_{n0}(x)$ holds, from the part (\ref{prev_point}) of the \cref{thm2}, we see that 
    \begin{equation}\label{wequation121} 
    u_{n+1}(x,t) \leq u_n(x,t)
    \end{equation}
Define 
\begin{equation}\label{equation123221}
    u(x,t) := \lim_{n \rightarrow \infty} u_n(x,t)
\end{equation} Furthermore, let $u_0$ and $\widetilde{u_0}$ be functions in $L^1(\re)$ and set $u(x,t)$ and $\widetilde{u}(x,t)$ to be as in \cref{equation123221}. Then, from \cref{equation100000}, it follows that
\begin{equation}\label{equation121}
\begin{split}
    \int_{-\infty}^{\infty}\int_{\tau}^T |u(x,t) - \widetilde{u}(x,t)| dx dt &\leq \lim_{n \rightarrow \infty} \int_{-\infty}^{\infty}\int_{\tau}^T |u_n(x,t) - \widetilde{u_n}(x,t)| dx dt\\
    &\leq (T-\tau) \lim_{n \rightarrow\infty}\int_{-\infty}^{\infty} |u_{n0}(x) - \widetilde{u_{n0}}(x)| dx \\
    &\leq (T-\tau) \int_{-\infty}^{\infty}|u_0(x) - \widetilde{u_0}(x)|dx.
\end{split}
\end{equation}

As in the earlier proof, we have
\[
\lim_{\tau \rightarrow T } \frac{1}{T - \tau}\int_{-\infty}^{\infty}\int_{\tau}^T |u(x,t) - \widetilde{u}(x,t)| dx dt = \int_{-\infty}^{\infty } |u(x,T) - \widetilde{u}(x,T)| dx dt.
\]
Along with \cref{equation121}, we see that \cref{call3.1} is established.
Now, as $u_{n0}(x) \geq -n$, we have that $u_n(x,t) \geq -n$ for a.e $(x,t) \in \re\times(0,\infty)$ and so there holds
\[f\left(u_n(x,t)\right) = f_n\left(u_n(x,t)\right).\]
Hence, for a compact set $K\subset\re\times(0,\infty)$ and for any $\varphi \in C_c^{\infty}(K)$, we have
\begin{equation}
        \int_{-\infty}^{\infty} \int_0^{\infty} \left( u_n \varphi_t + f(u_n)\varphi_x\right) dx dt =\underbrace{ \int_{-\infty}^{\infty} \int_0^{\infty} \left( u_n \varphi_t + f_n(u_n)\varphi_x\right) dx dt.}_{\text{equals }0}
\end{equation}

Since $u_- \leq 0 \leq u_+ = \infty$, $|u_-| <\infty$, there exist $\alpha > 0, \beta > 0$ such that
\begin{equation}\label{equationew124}
  f(p) \leq \left\{ 
  \begin{aligned}
  &\alpha + \beta |p| , \quad \text{ if } p \leq 0 \\
  & \alpha + f(p), \quad \text{ if } p\geq 0.
  \end{aligned}
  \right.
\end{equation}
Set 
\[ E_- := \{ (x,t) \in \re\times(0,\infty) ; u(x,t) \leq 0\}\]
\[ E_+ := \{ (x,t) \in \re\times(0,\infty) ; u(x,t) \geq 0\}\]

Let $K \subset \re \times (0,\infty)$ be a compact set. As in the proof of \cref{thm2}, since $u_+ = \infty$, there exist a constant $C(K)\geq 0$ such that

\[ |u_1(x,t) | \leq C(K), \quad \forall (x,t) \in E_+.\]

Now, $u(x,t) \leq u_1(x,t)$, we have 
\begin{equation}\label{equationew125}
    |u(x,t)| \leq C(K), \quad \text{ for $(x,t) \in E_+$.}
\end{equation} 

The \cref{equationew124} gives that for a.e $(x,t) \in \re \times (0,\infty)$, we have
\begin{equation}\label{equationew126}
  |f(u(x,t))| \leq \left\{ 
  \begin{aligned}
  &\alpha + \beta |u(x,t)| , \quad \text{ if } (x,t) \in E_- ,\\
  & \alpha + |f(u(x,t))|, \quad \text{ if } (x,t) \in E_+.
  \end{aligned}
  \right.
\end{equation}

The \cref{equationew125} tells that there exist a $\tau \geq \beta$ such that 
\begin{equation}\label{equationew127}
    |f(u(x,t))| \leq \alpha + \tau | u(x,t)|.
\end{equation}

The function $u$ is in $L^1(\re \times (0,T))$, for all $T>0$ tells by the Dominated Convergence theorem, that for all $\varphi \in C_c^{\infty}(K)$, there holds

\begin{equation}
\begin{split}
    \int_{-\infty}^{\infty} \int_0^{\infty} \left[ u \varphi_t + f(u) \varphi_x\right] dx dt  &= \lim_{n \rightarrow \infty} \int_{-\infty}^{\infty} \int_0^{\infty} \left[ u_n \varphi_t + f(u_n) \varphi_x\right] dx dt \\
    &= \lim_{n \rightarrow \infty} \int_{-\infty}^{\infty} \int_0^{\infty} \left[ u_n \varphi_t + f_n(u_n) \varphi_x\right] dx dt\\
    & = 0,
\end{split}
\end{equation}
 which is true by the fact that $f(u_n) = f_n(u_n)$.

Finally, as in the \cref{thm2}, for $a<b$, we have
\[\lim_{t\rightarrow0}\int_a^b u(x,t) dx = \int_a^b u_0(x) dx,\]
which concludes the proof for the third theorem.
\end{proof}

\section*{\bf Appendix.}

\begin{enumerate}
    \item{\bf{Mollification to Super Linear Growth.}}\label{app1} For $f :\re \mapsto \re$, a convex function, $f$ is differentiable a.e. Let $A<B$, be two points in $\re$ where $f$ is differentiable at. Let $D >0$ and set 
    \begin{equation}\label{equation117}
       g(x) := \left\{ \begin{aligned}
            &f(p) &\text{ if } A \leq p \leq B, \\
            &f(A) + f'(A) (p-A) + D(p-A)^2 &\text{ if } p\leq A, \\
            &f(B) + f'(B) (p-B) + D(p-B)^2 &\text{ if } p\geq B,
        \end{aligned}\right.
    \end{equation}
Then the function $g$ has the following properties:
\begin{itemize}
    \item The function $g$ has superlinear growth.
    \item The function $g$ is convex.
    \item There holds the equality $g(x) = f(x)$, for $x \in [A,B]$.
\end{itemize}
    \item \label{proflem1}
    \begin{proof}[\bf{Proof of the Lemma (\ref{lem1}).}]
    As in the definition of $f^*$,
    \[f^*(q) := \sup \{ p.q - f(p) ; p \in \re \},\]
    we see that for any $p,q \in \re$, there holds
    \[ f^*(q) \geq p.q - f(p).\]
    Normalising the quantities, we get
    \[ \frac{f^*(q)}{|q|} \geq p . \frac{q}{|q|} - \frac{f(p)}{|q|},\]
    which tells that
    \[ \lim_{|q| \rightarrow \infty} \frac{f^*(q)}{|q|} \geq p.w \quad \text{for all p}, \]
    where $w \in \{ -1, +1\}$. Sending $p.w$ to infinty gives the superlinearity of $f^*$,\[\lim_{|q| \rightarrow \infty} \frac{f^*(q)}{|q|} = \infty . \]
    The function $f$ is superlinear implies \begin{equation*}
        \begin{split}
            p.q - f(p) &= |p| \left ( \frac{p}{|p|}q - \frac{f(p)}{|p|}\right) \\
            &\leq \underbrace{|p| \left(|q| - \frac{f(p)}{|p|}\right),}_{ \text{goes to $-\infty$ as $|p|$ goes to infinity. }}
        \end{split}
    \end{equation*}
    which tells that there exists $p_0 \geq 0$ such that
    \begin{equation*}
        f^*(q) = \sup_{p \in \re} \{ p.q - f(p)\} \leq \sup_{|p| \leq p_0} \{ p.q - f(p)\},
    \end{equation*}
    and so, $f^*(q) < \infty$, for all $q \in \re$.
    
    Let $f_{\epsilon}$ satisfy the assumptions (\ref{it23}). Let $M>0$ and $|q| \leq M$, then we have
    \begin{equation}
        \begin{split}
            p.q - \fe(p) &=  |p| \left ( \frac{p}{|p|}q - \frac{\fe(p)}{|p|}\right) \\
            &\leq  |p| \left ( M - \frac{\fe(p)}{|p|}\right),
        \end{split}
    \end{equation}
    and so we see that
    \begin{equation*}
        \lim_{|p| \rightarrow \infty} \sup_{0 < \epsilon \leq 1}\{ p.q - \fe(p)\}\leq \lim_{|p| < \infty} |p| \left\{ M - \inf_{0 < \epsilon \leq 1}\frac{\fe(p)}{|p|}\right\}.
    \end{equation*}
    Thus, there exists $p_0(M)$ independent of $\epsilon$, such that for all $\epsilon \in (0,1]$, $|q| \leq M$, there holds
    \[\fe^*(q) = \sup_{|p| \leq p_0}\left\{p.q - \fe (p)\right\}.\]
    Again as $f$ is superlinear, by similar arguement, there exists $p_1 > 0$
    \[f^*(q) = \sup_{|p| \leq p_1}\left\{p.q - f (p)\right\}.\]
    Set $p_2 := \max\{ p_0, p_1\}$. Then, for $|q| \leq M$, we have
    \[\fe^*(q) = \sup_{|p| \leq p_2}\left\{p.q - \fe (p)\right\},\]
    \[f^*(q) = \sup_{|p| \leq p_2}\left\{p.q - f (p)\right\}.\]
    So, by continuity and compactness, for all $|q| \leq M$, there exists $q_1 = q_1(q), q_2 = q_2(q)$ such that $|q_1| \leq p_2 $, $|q_2| \leq p_2$ and 
    \[\fe^*(q) = q_1 q - \fe(q_1),\]
    \[ f(q) = q_2 q - f(q_2).\]
    Hence, for all $p \in \re$, we have
    \[\fe^*(q) - f(q) \leq q_1 q -\fe(q_1) - pq +f(p).\]
    Setting $p = q_1$, we have
    \[\fe^*(q) - f^*(q)\leq f(q_1) - \fe(q_1) \leq \sup_{|p| \leq p_2}|f(p) - \fe(p)|.\]
    Interchanging $f \leftrightarrow \fe$, for all $|q| \leq M$, there holds
    \[ |\fe^*(q) - f^*(q)| \leq  \sup_{|p| \leq p_2}|f(p) - \fe(p)|.\]
    Hence, $\fe^*$ converges to $f^*$ on compact sets uniformly. Furthermore, for all $p \in \re$, we have
    \[\frac{\inf_{0<\epsilon\leq 1}\fe^*(q)}{|q|} \geq \frac{pq}{|q|}- \sup_{0 < \epsilon \leq 1}\frac{\fe^*(p)}{|q|}. \]
    So, for $w \in \{-1,+1\}$, with $\frac{q}{|q|} \rightarrow w$ says
    \[\lim_{|q| \rightarrow \infty}\inf_{0 < \epsilon \leq 1}\frac{\fe^*(q)}{|q|} \geq pw. \]
    Letting $pw$ to go to infinity, we obatin
    \[ \lim_{|q| \rightarrow \infty}\inf_{0 < \epsilon \leq 1}\fe^*(q) = \infty.\]
    This proves the first three parts of the lemma. For the last part of the lemma, define a new function
    \[ F_{\epsilon}(x) := \left(\alpha_{\epsilon}*F\right)(x) + \epsilon x^2.\]
    The function $F\ast \alpha_{\epsilon}$ is smooth and convex as $F$ is convex and hence $(F\ast \alpha_{\epsilon})'' \geq 0$. Equivalently, there holds
    \[F_{\epsilon}''(x) \geq \epsilon >0,\]
    i.e $F_{\epsilon}$ is uniformly convex. Let $\alpha \in C_c^{\infty}(B(0,1))$. As $F$ as has superlinear growth and is convex, there exist $q_0 >0$ such that
    \begin{itemize}
        \item $F(p) > 0$, for all $|p| \geq q_0$.
        \item The function $F(p)$ is non decreasing for $p > q_0$.
        \item The function $F(p)$ is non increasing for $p < - q_0$
    \end{itemize}
    Hence, for all $x \geq q_0 + 1$, for all $0<\epsilon \leq 1$, $|y|\leq 1$, we have
    \[ F(x - \epsilon y) \geq F(x-1),\]
    and for all $x \leq -q_0 - 1$, for all $0<\epsilon \leq 1$, $|y|\leq 1$, we have
    \[ F(x - \epsilon y) \geq F(x+1).\]
    Taking the limits, we get
    \[ \lim_{x \rightarrow\infty}\inf_{0<\epsilon \leq 1} \frac{F(x - \epsilon y )}{|x \pm 1|} \geq \lim_{x \rightarrow \infty } \frac{F(x \pm 1)}{|x \pm 1|} = \infty. \]
    So, there holds
    \begin{equation*}
        \begin{split}
            \lim_{x \rightarrow\infty}\inf_{0<\epsilon \leq 1}\left(\frac{F_{\epsilon}(x)}{x}\right) &\geq \lim_{x \rightarrow\infty}\inf_{0<\epsilon \leq 1} \int_{|y| \leq 1} \frac{F(x - \epsilon y)}{x} \alpha(y) dy \\
            &\geq  \underbrace{\lim_{x \rightarrow\infty}\int_{|y| \leq 1} \frac{F(x-1)}{|x-1|}\frac{|x-1|}{|x|}\alpha(y) dy.}_{ = \infty}
        \end{split}
    \end{equation*}
    Similarly, we get 
    \[ \lim_{x \rightarrow\infty}\inf_{0<\epsilon \leq 1}\left(\frac{F_{\epsilon}(x)}{-x}\right) \leq  \underbrace{\lim_{x \rightarrow\infty}\int_{|y| \leq 1} \frac{F(x-1)}{|x-1|}\frac{|x-1|}{|x|}\alpha(y) dy.}_{ = \infty}\]
    This concludes the proof for the lemma.
    
    \end{proof}

\item {\begin{Proposition}\label{prop_new}
     Let $\Omega \subset \re^n$ be a bounded open set. Furthermore, let $\{ w_k\} \subset L^{\infty}(\Omega)$ and let $w \in L^{\infty}(\Omega)$. Also, let $M_1 >0$ and $M_2 >0$ be two constants such that for all $k \in \mathbb{N}$, there holds
    \begin{itemize}
        \item\label{prop_new1.1} $\lVert w_k \rVert_{L^{\infty}(\Omega)} \leq M_1,$
        \item\label{prop_new1.2} $\lVert w \rVert_{L^{\infty}(\Omega)} \leq M_2$.
    \end{itemize}
    Moreover assume $\varphi \in C_c^1(\Omega)$, there holds
    \begin{equation}\label{prop_new2}
    \lim_{k \rightarrow \infty} \int_{\Omega} w_k(x) \varphi(x) dx = \int_{\Omega} w(x)\varphi(x) dx.
    \end{equation}

Then, we see that
\begin{equation}\label{prop_new3}
    \lVert w \rVert_{L^{\infty}(\Omega)} \leq M_1.
\end{equation}
(Also, refer \cite{Rudin_functional}).
\end{Proposition}
\begin{proof}
    By the regularity of the Lebesgue measure, we have that the space $C_c^1(\Omega)$ is dense in the space $L^1(\Omega)$. Hence, for all $f \in L^1(\Omega)$ and for all $\eta >0$, there exist $\varphi \in C_c^1(\Omega)$ such that 
    \begin{equation}
        \int_{\Omega} | f-\varphi| dx < \eta.
    \end{equation}
    Now, from the hypothesis of the proposition, we see that
    \begin{equation}
        \begin{split}
            \left| \int_{\Omega} \left(w_k - w\right) f dx \right| &= \left| \int_{\Omega}\left(w_k - w\right) \varphi dx  \right| + \left| \int_{\Omega} \left( w-w_k\right) \left(f - \varphi\right) dx \right| \\
            & \leq \left|\int_{\Omega} \left(w_k - w\right) \varphi dx  \right| + (M_1 + M_2) \lVert f - \varphi \rVert_{L^1(\Omega)} \\
            & \leq \left| \int_{\Omega} w_k \varphi - \int_{\Omega} w \varphi \right| + \eta (M_1 + M_2).
        \end{split}
    \end{equation}
Now, sending $k \rightarrow \infty$, from \cref{prop_new2}, we obtain 
\begin{equation}\label{prop_new5}
    \lim_{k \rightarrow \infty} \int_{\Omega} w_k f = \int_{\Omega} w f dx.
\end{equation}
Define new functions $l_k$ and $l$ in the dual space $L^1(\omega)^*$ by
\[
    l_k(f) := \int_{\Omega} w_k f dx \quad \text{and} \quad l(f) := \int_{\Omega} w f dx.
\]
Then, from \cref{prop_new5} and the hypothesis, we have
\begin{itemize}
    \item $|l_k(f)| \leq M_1 \lVert f \rVert_{L^1(\Omega)}$, equivalently, the operator norm $\lVert l_k \rVert$ is bounded by $M_1$,
    \item $|l(f)| \leq M_2 \lVert f \rVert_{L^1(\Omega)}$, equivalently, the operator norm $\lVert l \rVert$ is bounded by $M_2$,
    \item $l_k(f) \rightarrow l(f)$, for all $f$ in $L^1(\Omega)$, i.e $l_k$ converges to $l$ weakly in $L^1(\Omega)^*$.
\end{itemize}

Now, Banach-Alaoglu's theorem tells that the closed ball $\overline{B(0,M_1)}$ in $L^1(\Omega)^*$ is weakly compact. As $l_k$ is in $\overline{B(0,M_1)}$, for all $k$, we have that $l \in \overline{B(0,M_1)}$,  which concludes that
\begin{equation*}
\lVert w \rVert_{L^{\infty}(\Omega)} \leq M_1.
\end{equation*}
\end{proof}}
\end{enumerate}

\bibliographystyle{alpha}
\bibliography{refer}

\end{document}